\documentclass{article}
\usepackage{graphicx} 
\usepackage{xcolor}
\usepackage{hyperref}
\hypersetup{
colorlinks = true, 
urlcolor = blue, 
linkcolor = red, 
citecolor = blue 
}

\usepackage{enumitem}
\usepackage{bbm}
\usepackage{caption}
\usepackage{authblk}
\usepackage{amssymb}
\usepackage{amsthm}
\usepackage{amsmath} 
\usepackage{imakeidx}
\usepackage{seqsplit}
\makeindex[title=Notation Index, intoc,columns=1]
\indexsetup{othercode=\noindent\par}

\usepackage{thmtools, thm-restate}
\usepackage{stmaryrd}
\usepackage{mathtools}
\usepackage{tikz}
\usetikzlibrary{decorations.pathreplacing}
\usetikzlibrary{patterns}
\usepackage[normalem]{ulem}

\declaretheorem[numberwithin=section]{theorem}
\newtheorem{lemma}[theorem]{Lemma}

\newtheorem{question}[theorem]{Question}
\newtheorem{definition}[theorem]{Definition}
\newtheorem{example}[theorem]{Example}

\newtheorem{notation}[theorem]{Notation}
\newtheorem{proposition}[theorem]{Proposition}
\newtheorem{corollary}[theorem]{Corollary}
\newtheorem{remark}[theorem]{Remark}

\newcommand{\Z}{\mathbb Z}
\newcommand{\N}{\mathbb N}

\newcommand{\notationidx}[2]{#1\index{#1@#1 --- #2|hyperpage}}

\title{On the cohomology of homshifts}

\author[1]{\normalsize Nishant Chandgotia} 

\affil[1]{ \small
	Tata Institute of Fundamental Research - Centre for Applicable Mathematics
    \url{nishant.chandgotia@gmail.com}
	}

\author[3]{\normalsize Silvère Gangloff}

\author[2]{\normalsize Benjamin Hellouin de Menibus}

\affil[2]{
Laboratoire Interdisciplinaire des Sciences du Numérique, Université Paris-Saclay
\url{hellouin@lisn.fr}}

\author[3]{Piotr Oprocha}

\affil[3]{
National Supercomputing Centre IT4Innovations, IRAFM, University of Ostrava,
	30. dubna 22, 70103 Ostrava,
	Czech Republic\\
    \url{silvere.gangloff@osu.cz}\\
    \url{piotr.oprocha@osu.cz}
	}

\date{\today}

\begin{document}

\maketitle
\begin{abstract}
  We study the cohomology of symbolic dynamical systems called homshifts: they are the nearest-neighbour $\Z^d$ shifts of finite type whose adjacency rules are the same in every direction. Building on the work of Klaus Schmidt (Pacific J. Math. 170 (1995), no. 1, 237–269) we give a necessary and sufficient condition for homshifts to be cohomologically trivial. This condition is expressed in terms of the topology of a natural two-dimensional CW complex arising from the shift space which can be analyzed in many natural cases. However, we prove that in general, cohomological triviality is algorithmically undecidable for homshifts.
\end{abstract}
 \bigskip 

\noindent \textbf{Mathematics Subject Classification}. \emph{Primary : 37B51 (Multidimensional shifts of finite type); Secondary :  05C25 (Graphs and groups), 37A20 (Orbit equivalence, cocycles, ergodic equivalence relations)} 
 { \hypersetup{linkcolor=black} \tableofcontents}

\section{Introduction}

The study of symbolic dynamics has long been intertwined with questions in geometry, group theory, and statistical physics. In the early 1990s, tiling problems played a particularly influential role in shaping these connections. Conway and Lagarias \cite{MR1041445} introduced group-theoretic invariants to study whether a given finite set of polygonal shapes could tile regions of the plane, and Thurston \cite{MR1072815} further developed these ideas, popularizing what are now referred to as the Conway–Lagarias–Thurston tiling groups. In parallel, Katok and Spatzier \cite{MR1307298} investigated cohomological rigidity phenomena for higher-rank abelian group actions, showing how cohomological triviality captures fundamental dynamical constraints.

Klaus Schmidt \cite{Schmidt95} later recognized that the groups introduced by Conway and Lagarias could be interpreted in terms of cocycles on certain 
$\mathbb Z^2$-subshifts of finite type. This provided a striking link: the algebraic structure underlying tiling groups reappears naturally in the cohomology of shift spaces, and rigidity phenomena studied by Katok and Spatzier in smooth dynamics find discrete analogues in symbolic dynamics. In subsequent work, Schmidt \cite{Schmidt95} and Einsiedler \cite{MR1836431} initiated a systematic study of cocycles and cohomology for higher-dimensional subshifts of finite type. A central motivating question in this line of work is whether cohomology can be used to decide extension problems: given a configuration defined outside a finite region, can it be extended inwards to a global configuration of the subshift?

In the particular case of lozenge tilings and dimer tilings, Thurston gave a necessary and sufficient condition for a region to be tiled \cite{MR1072815}. This leads to similar conditions on more general tilings which relates the geometry of the tiling group to a nice necessary and sufficient condition via the so-called Kirszbraun/McShane Whitney-type theorem \cite{tassy2014tiling, chandgotia2019kirszbraun}. A nice presentation of these results can be found in \cite[Section 8]{DeyThesis2026}.
On the other hand, the same tools also seem to play an important role in questions about large scale statistical physics phenomena such as what does a uniformly sampled tiling look like \cite{tassy2014tiling,menz2020variational, zbMATH07803200}.

In this work, we define analogous groups for a particular class of statistical physics models arising from graph homomorphisms on the cubical lattice called homshifts. Given a finite connected undirected graph $G$, the $d$-dimensional homshift $X^d_G$ is defined as the space of graph homomorphisms from the Cayley graph of $\Z^d$
 (with respect to its standard generators) into 
$G$. These $d$-dimensional homshifts, for $d>1$, form a rich subclass of $\Z^d$-shifts of finite type. Their definition is simple and highly symmetric and many meaningful properties are algorithmically decidable: for instance, non-emptiness is decidable and entropy is algorithmically computable \cite{MR1428636}. This contrasts with shifts of finite type where the same problems are undecidable \cite{MR216954, MR2680402}. As a matter of fact, it is the case for most nontrivial properties \cite{carrasco2025rice} {for general shifts of finite type}. Furthermore, they encompass many models of interest in statistical physics, such as the hard-core model, proper colorings, and the iceberg model, and some finer properties still exhibit undecidability phenomena \cite{gao2018continuous, CGHO25}. They provide a particularly fertile ground for exploring the boundary between tractable and intractable questions in symbolic dynamics.

The main result of the present paper relates the cohomology of homshifts with the even square group, which is an index $2$ subgroup of the square group introduced in \cite{CGHO25,gao2018continuous}. 
\begin{theorem}\label{theorem: cohomological triviality}
    For all $d > 1$ and for any graph $G$ such that $X^d_G$ is topologically mixing, $X^d_G$ is cohomologically trivial if and only if  the square group is isomorphic to $\Z/2\Z$ if and only if the even square group of $G$ is trivial.
\end{theorem}
Note that $X^d_G$ is mixing if and only if $G$ is a connected graph which is not bipartite. The mixing hypothesis is necessary to talk about cohomology given the standard definitions (see Section \ref{section.non.mixing}). 

 The second equivalence is almost immediate from the definition (see Proposition~\ref{Prop: basic_properties_square}). For the direction $(\Rightarrow)$ of the first equivalence, we use the action of the square group on the square cover to define a natural cocycle on the shift space taking values in the square group, and prove that this cocycle is nontrivial when the even square group is nontrivial (Proposition \ref{proposition.non.trivial.sqgrp.cocycle}).

The proof of the direction $(\Leftarrow)$ of the first equivalence relies mainly on techniques introduced by Schmidt \cite{Schmidt95} to show that two-dimensional shifts of finite type with a specification property admit no nontrivial cocycle. A similar specification-like property has been defined by Chung and Jiang \cite{zbMATH06774915} to prove triviality of cohomology in the context of shifts on groups more general than $\Z^d$. We introduce a strip-gluing condition (Definition~\ref{def:stripgluing}) which we prove to be a consequence of having a trivial even square group. Thereafter we adapt Schmidt's ideas to conclude cohomological triviality for strip-gluing homshifts. An important difference is that Schmidt's results apply to cocycles taking values in locally compact second countable groups while we focus on cocycles taking values in discrete groups. Our results further extend to higher dimensions.

However, it is undecidable whether the even square group is trivial. Therefore we get the following result.
\begin{theorem}\label{thm.undecidability.cocycle}
For all $d > 1$, provided a graph $G$ such that $X^d_G$ is topologically mixing, it is not possible to decide if $X^d_G$ has a nontrivial cocycle or not.
\end{theorem}

This problem was known to be undecidable in shifts of finite type for generic reasons \cite{carrasco2025rice} (Rice-type theorem); we stress again that this kind of theorem cannot apply to homshifts. In general, undecidability results for homshifts require finer considerations.

While the properties of the square group (and hence the even square group) are undecidable (following \cite{gao2018continuous,CGHO25}), it has an explicit and natural description which is computable under reasonable assumptions: if the graph $G$ has no self-loops the square group is the fundamental group of the complex obtained by adding a square 2-cell $[0,1]^2$ on each `square' appearing in $G$. This construction can be easily adapted to the case when the graph does have self-loops by taking a bipartite cover first.

From this description, it follows that if the resulting complex is a surface or four-cycle free, we can compute the square group and thus the even square group quite explicitly. In both cases, since the resulting square group is hyperbolic, it opens applications of noncommutative ergodic theorems like the one of Karlsson and Margulis \cite{zbMATH01415732}. We hope that this will help us in the study of large scale phenomena in random graph homomorphisms in the future. 

We notice that the square group is reminiscent of the projective fundamental group introduced by Geller and Propp \cite{zbMATH00838593}. In the restricted case of four-cycle-free non-bipartite graphs, Paviet-Salomon identifies the projective fundamental group of the corresponding homshifts in his Ph.D. thesis \cite[Theorem \textbf{3.76}]{salomon2024iotandecidabilite}, which is the fundamental group of the graph in this case.

The article is organized as follows. After some background definitions in Section \ref{section.background}, the even square group is introduced in Section \ref{section.even.square.group}. In Section \ref{section.technical.results}, we prove a useful technical characterization of square-decomposability (notion introduced in \cite{gangloff2022short}) in terms of pattern completion. Section \ref{section.sq.cocycle} defines the square group cocycle and uses it to prove direction $(\Rightarrow)$ in Theorem \ref{theorem: cohomological triviality}.
In Section~\ref{section.cocycle.infinite} we prove the direction $(\Leftarrow)$ in Theorem \ref{theorem: cohomological triviality} for two-dimensional homshifts ($d=2$) and generalize it to higher dimensions $(d > 2)$  in Section~\ref{section.cohomology.higher.dim}. In Section~\ref{section.boxext}, we prove that the box-extension property, which implies cohomological triviality in general \cite{Schmidt95}, is strictly stronger, even for homshifts. Finally, we discuss cohomological triviality for non-mixing subshifts in Section \ref{section.non.mixing}.

\section*{Acknowledgments}
We would like to thank Tom Meyerovitch and Klaus Schmidt for helping us along the way, especially regarding the cohomology of non-mixing spaces. We thank Ville Salo for introducing us to the work by Chung and Jiang \cite{zbMATH06774915}. The third author is grateful to Léo Paviet-Salomon for various discussions and for introducing him to the projective fundamental group. Research of N.\ Chandgotia was partially supported by SERB SRG grant and INSA fellowship and acknowledges the support of the Department of Atomic Energy, Government of India, under Project Identification No.\ RTI 4014.

This research was partially supported by EU funds under the project ``Increasing the resilience of power
grids in the context of decarbonisation, decentralisation and sustainable socioeconomic development'',
CZ.02.01.01/00/23\_021/0008759, through the Operational Programme Johannes Amos Comenius. 

\section*{Tool and computational resource disclosure}
No large language model or similar tool has been used in the course of this research or in writing the current article.

\section{Background\label{section.background}}

{Throughout the article, $d>1$ refers to the dimension.}
For all $\boldsymbol{n} \in \mathbb{Z}^d$, we denote 
by \notationidx{$\lvert\boldsymbol{n}\rvert_1$}{$\ell^1$ norm of vector $\boldsymbol{n} \in \mathbb Z^d$} the $\ell^1$ norm $\sum_i |\boldsymbol{n}_i|$. For all $r \ge 0$, we set \notationidx{$\boldsymbol{B}^d(r)$}{set $\llbracket -r, r \rrbracket ^d$} $\coloneqq \llbracket -r, r \rrbracket ^d$ and denote by \notationidx{$\partial \boldsymbol{B}^d(r)$}{border of the box $\boldsymbol{B}(r)$} 
the set $\boldsymbol{B}^d(r) \backslash \boldsymbol{B}^d({r}-1)$. 

\subsection{Graphs}

For every graph $G$, we denote by \notationidx{$V_G$}{vertex set of the graph $G$} the set of its vertices and \notationidx{$E_G$}{edge set of the graph $G$} the set of its edges, where an edge is a tuple $(u,v)$ with $u,v \in V_G$. We say that the graph is \textbf{undirected} when, for all $u,v$ vertices in $V_G$, $(u,v) \in E_G$ iff $(v,u) \in E_G$. We say that $G$ is \textbf{finite} when both $V_G$ and $E_G$ are finite. \bigskip 

A \textbf{walk} on $G$ is a finite word $p = p_0 \cdots p_n$ of vertices of $G$ such that $(p_i,p_{i+1})\in E_G$ for all $i < n$. A walk is said to be \textbf{simple}
when it does not take twice the same value. The integer $n$ is called the \textbf{length} of $p$ and is denoted by \notationidx{$l(p)$}{length of a walk $p$}. Such a walk $p$ is called a \textbf{cycle} 
when $p_{l(p)} = p_0$. Such a cycle is said to be \textbf{simple} when for all $i<j$, if $p_i = p_j$ then $i=0$ and $j=l(p)$. 
For any cycle $p$, we will denote by \notationidx{$\omega(p)$}{circular shift of a walk $p$} the \textbf{circular shift} of $p$, defined as $\omega(p) = p_1 \ldots p_{l(p)} p_1$. A cycle of length two is called a \textbf{backtrack}. A walk which does not contain any backtrack is called non-backtracking. A \textbf{square} is a non-backtracking cycle of length four. 
A walk is said to be \textbf{trivial} when it has length zero.
We say that a graph $G$ is connected when for all $u,v \in G$, there exists a walk $p$ on $G$ such that $p_0 = u$ and $p_{l(p)} = v$. 

    For any two walks $p,q$ such that $p_{l(p)} = q_0$, we denote by \notationidx{$p \odot q$}{`concatenation` operation on walks} the walk $p_0 \ldots p_{l(p)} q_1 \ldots q_{l(q)}$ and by $p^{-1}$ the reverse walk $p_{l(p)} \ldots p_0$. We denote by \notationidx{$\varphi$}{map which removes backtracks from walks} the function on the set of walks on $G$ such that for all walks $p$, $\varphi(p)$ is obtained from $p$ by replacing successively all backtracks $aba$ by $a$. The order of removal does not change the non-backtracking walk obtained at the end, so $\varphi$ is well-defined. The proof is very similar to that used to define a free group (see for instance \cite[Chapter 6.7]{zbMATH00425998}) and shall be skipped. For all walks $p,q$ such that $p_{l(p)} = q_0$, we set \notationidx{$p \star q$}{word obtained by removing backtracks from the walk concatenation of $p$ and $q$} $\coloneqq\varphi(p \odot q)$. The operation $\star$ is associative. 

{A \textbf{tree} is a graph that contains no nontrivial simple cycle}. A \textbf{spanning tree} of an undirected graph $G$ is a subgraph $T$ of $G$ {which is a tree such that} $V_T = V_G$.\bigskip

\textit{In the remainder of this article, graphs denoted by $G$ are assumed to be finite, undirected and connected.}

\subsection{Subshifts}

\subsubsection*{Elementary definitions}

Given a finite set $\mathcal{A}$ endowed with the discrete topology, we endow $\mathcal{A}^{\Z^d}$ with the corresponding product topology. 
Elements of $\mathcal{A}^{\Z^d}$ are called \textbf{configurations}.
The group $\Z^d$ acts on this space naturally by translations on configurations, via the \textbf{shift action} \notationidx{$\sigma$}{shift action}, defined as follows: for every $\textbf{i} \in \Z^d$ and $x\in X$, we denote by $\sigma^{\textbf{i}}(x)\in \mathcal{A}^{\Z^d}$ the configuration given by
\[(\sigma^{\textbf{i}}(x))_{\textbf{j}}=x_{\textbf{i}+\textbf{j}}\text{ for all }\textbf{j} \in \mathbb{Z}^d.\]

Each of the maps $\sigma^{\textbf{i}}$ is, in fact, a homeomorphism on $\mathcal{A}^{\Z^d}$. A \textbf{subshift} is a closed subset $X\subset \mathcal{A}^{\Z^d}$ which is invariant under the shift action. Alternatively, one may define subshifts by means of forbidden patterns \cite[Section \textbf{13.10}]{MR4412543} as {follows}. Given a finite set $\mathbb{U} \subset \Z^d$, a \textbf{pattern} on $\mathbb{U}$ is an element of $\mathcal{A}^\mathbb{U}$. 
Given a set of patterns $\mathcal F$ we denote by 
\[X_{\mathcal{F}}\coloneqq\{x\in \mathcal{A}^{\Z^d}~:~\text{ no translate of a pattern in $\mathcal F$ appears in $x$}\}.\]
It is well-known that $X\subset \mathcal{A}^{\Z^d}$ is a subshift if and only if there exists a set of forbidden patterns $\mathcal F\subset \mathcal{A}^{\star}$ such that $X=X_\mathcal F$; see \cite[Chapter \textbf{6}]{MR4412543} for the proof when $d=1$. The proof is similar in higher dimensions. A globally admissible pattern of $X$ is a pattern which appears in at least one configuration of $X$.
Of special interest are subshifts of \textbf{finite type}, which are subshifts $X_{\mathcal{F}}$ when the set $\mathcal F$ is finite.

\subsubsection*{Homshifts}

In this paper, we focus on a special class of shifts of finite type, called homshifts, which are more tractable than general shifts of finite type, while preserving some undecidability for much finer properties.

For two graphs $G$ and $H$, a \emph{graph homomorphism} $\phi: G\to H$ is a map from $V_G$ to $V_H$ which preserves adjacency, meaning that if $(g_1, g_2)\in E_G$ then $(\phi(g_1), \phi(g_2))\in E_H$. From here on, we will think of $\Z^d$ as both a group and as the $d$-dimensional grid graph, that is, the Cayley graph of this group with standard generators. Given a graph $G$, the $d$-dimensional \textbf{homshift} on $G$ is the subshift of finite type \notationidx{$X^d_G$}{$d$-dimensional homshift associated with $G$} whose configurations are the graph homomorphisms from $\Z^d$ to $G$. Here are a few important examples.

\begin{enumerate}
    \item \textit{Proper colourings}. For $n \in \mathbb{N}$, denote by $K_n$ the complete graph on $\{1,2, \ldots, n\}$, whose edges are the $(i,j)$ such that $i\neq j$. Then $X^d_{K_n}$ is the space of proper $n$-colourings of $\Z^d$, that is, assignments of colours $1, 2, \ldots, n$ to vertices of the graph $\mathbb{Z}^d$ such that colours on adjacent vertices are distinct.
    \item \textit{Hard core model}. Let $G$ be the graph on $\{0,1\}$ with one edge between $0$ and $1$ and one self-loop on $0$. The homshift $X^d_G$ is the set of configurations where no two adjacent $1$'s are allowed and is called the hard-core model.
    \item \textit{Iceberg model}. Fix $M>1$ and consider
    the graph with vertices {$\{i\in \Z~:~1\leq |i| \leq M\}$} and edges $\{(i,j)~:~ij<-1\}$. The corresponding homshift is the iceberg model introduced by Burton and Steif \cite{MR1279469}.
\end{enumerate}

{A locally admissible pattern of a homshift is a homomorphism from a finite subgraph of $\mathbb Z^d$ to $G$.} Every locally admissible pattern on a box $\boldsymbol{B}^d(\boldsymbol{n})$ is globally admissible (Proposition \textbf{2.1} in \cite{MR3743365}). Furthermore, {topological mixing is easy to decide for homshifts:}

\begin{lemma}[Proposition \textbf{3.1} in \cite{MR3743365}]\label{lemma.mixing.bipartite}
        {For every graph $G$, $X^d_G$ is topologically mixing if and only if $G$ is not bipartite.}
\end{lemma}

\subsubsection*{Cohomology}
 A \textbf{cocycle} \cite{Schmidt95} on a $d$-dimensional subshift $X$ is a map $c : \mathbb{Z}^d \times X \rightarrow \mathbb{G}$, where $\mathbb{G}$ is a discrete group, {which satisfies the \textbf{cocycle equation}}: for all $\boldsymbol{m},\boldsymbol{n} \in \mathbb{Z}^d$ and for all $z \in X$,
    \begin{equation}\label{eq;cocycle-def}c(\boldsymbol{m+n},z) = c(\boldsymbol{m},\sigma^{\boldsymbol{n}}(z)) c(\boldsymbol{n},z) = c(\boldsymbol{n},\sigma^{\boldsymbol{m}}(z)) c(\boldsymbol{m},z).
    \end{equation}
    A cocycle is said to be \textbf{continuous} when for all $\boldsymbol{n} \in \mathbb{Z}^d$, $c(\boldsymbol{n},\cdot)$ is continuous. {We consider two cocycles $c : \mathbb Z^d \times X \rightarrow \mathbb G$ and $c' : \mathbb Z^d \times X \rightarrow \mathbb G'$ to be identical whenever there is an isomorphism $\beta : \mathbb G \rightarrow \mathbb G'$ such that $c' = \beta \circ c$.} Two cocycles $c,c' : \mathbb{Z}^d \times X \rightarrow \mathbb{G}$ are said to be continuously \textbf{cohomologous} when there exists a continuous map $b: X \rightarrow \mathbb{G}$, called a \textbf{transfer function}, such that for all $\boldsymbol{n} \in \mathbb{Z}^d$, and $z \in X$,
    \[c(\boldsymbol{n},z) ={b(\sigma^{\boldsymbol{n}}(z))}^{-1} c'(\boldsymbol{n},z) b(z).\]
    A cocycle $c$ is called \textbf{group homomorphism} when $c(\boldsymbol{n},\cdot)$ is constant for all $\boldsymbol{n}$. In this case, we drop $z$ in the notations and write $c(\boldsymbol{n})$ instead of $c(\boldsymbol{n},z)$.
    We say that a cocycle is \textbf{trivial} when it is cohomologous to a group homomorphism.  Furthermore, we say that a subshift $X$ is \textbf{cohomologically trivial} when all its continuous cocycles are trivial.

\begin{remark}
As a consequence of the cocycle equation, one can compute $c(\boldsymbol{n},x)$ from the values of $c( \pm \boldsymbol{e}^i,\cdot)$. For any walk $w=w_0 w_1 w_2 \ldots w_r$ from $0$ to $\boldsymbol{n}$, we have:
\begin{equation}\label{equation.product.along.walk}
    c(\boldsymbol{n},x) = \prod_{i=r-1}^{0} c(w_{i+1}-w_i, \sigma^{w_i}(x)),
\end{equation}
where the product is applied from left to right in decreasing index order, so that $\prod_{i=r-1}^{0} a_i$ means $a_{r-1}a_{r-2}\cdots a_0$. This value is independent of the choice of walk $w$.

\end{remark}

\begin{remark}\label{remark.nullvalue.cocycle}
    Note also that for every cocycle $c : \mathbb Z ^d \times X \rightarrow \mathbb G$ and every configuration $x \in X$, the cocycle equation written with $\boldsymbol{n} = \boldsymbol{m} = \boldsymbol{0}$ yields 
    $c(\boldsymbol{0},x) = c(\boldsymbol{0},x) c(\boldsymbol{0},x)$. Dividing by $c(\boldsymbol{0},x)$, we obtain $c(\boldsymbol{0},x) = 1_{\mathbb G}$. Furthermore, for all $\boldsymbol{n}$, 
    $c(-\boldsymbol{n},x) = c(\boldsymbol{n},\sigma^{-n}(x))^{-1}$, as by the cocycle equation, we have:
    \[1_{\mathbb G} = c(\boldsymbol{0},x) = c(\boldsymbol{n} - \boldsymbol{n},x) = c(\boldsymbol{n},\sigma^{-\boldsymbol{n}}(x))c(-\boldsymbol{n},x).\]
\end{remark}

\paragraph{Motivation}

When defining cocycles and what it means for a cocycle to be trivial, Schmidt was inspired by groups which were introduced by Conway and  Lagarias \cite{MR1041445,MR1072815} in the context of tilings, {and the idea of deciding which regions can be tiled by a given set of tiles using maps to these groups.} 
In our context, this question is similar to the one of which homomorphisms on the boundary of a box can be extended inside the box.
Let us see how cocycles may be used to answer this question. 
 Consider a subshift $X$ and a cocycle $c$ on $X$. Assume that for $i= 1,2$, the maps $c( \pm \boldsymbol{e}^i,\cdot)$ are $r$-block maps.  For any pattern $q$ {on $\llbracket-r,r\rrbracket^2$ which can be extended to a configuration $x\in X$ }we can define $c(\pm \boldsymbol{e}^i,q)$ as $c(\pm \boldsymbol{e}^i,x)$.  Take a pattern $p$ on support $\partial\llbracket 0,n\rrbracket^2 + \llbracket -r,r\rrbracket^2$. 
  For $w$ and $w'$ the  shortest walks from $\boldsymbol{0}$ to $\boldsymbol{n}$ on the boundary of the box - respectively going up then right and going right then up - if we have 
 \[\prod_{i=r-1}^{0} c(w_{i+1}-w_i, p_{w_i + \llbracket -r,r\rrbracket^2}) \neq \prod_{i=r-1}^{0} c(w'_{i+1}-w'_i, p_{w'_i + \llbracket -r,r\rrbracket^2}),\] 
 then the pattern $p$ cannot be extended to a configuration in $X$ (otherwise \eqref{equation.product.along.walk} would not be satisfied for this configuration).

On the other hand, if $p$ can be extended to a configuration in $X$, then the two values must be equal. 
While this is not a sufficient condition, it is necessary. In many interesting cases, we can use this {and some additional ideas} to get a necessary and sufficient condition \cite{MR1072815,tassy2014tiling}. Note that the problem disappears if the cocycle is cohomologous to a group homomorphism. Indeed, suppose a cocycle $c$ takes values in a discrete group $\mathbb G$ and is cohomologous to a group homomorphism $c'$ for a transfer function $b$. Further assume that both $c(\boldsymbol{e}^i, \cdot)$ for $i=1,2$ and the transfer function $b$ are all $r$-block maps. Then we have that $c(n\boldsymbol{e}^1+n\boldsymbol{e}^2,x)$ equals
\[ (b(\sigma^{n\boldsymbol{e}^1+n\boldsymbol{e}^2}(x)))^{-1}b(x).\]
In other words, the cocycle will give us no information whatsoever whether $x$ can be extended to a graph homomorphism. 

\begin{example}
    Consider the subshift of three-colorings, which is the homshift associated with the graph $K_3$, the clique on three vertices $0,1,2$. Define a cocycle $c:  \Z^2\times X^2_{K_3} \to \Z$, by defining it for $e = (1,0)$ or $e=(0,1)$ by:
    \[c(e, x) = 
    \begin{cases}
    1 & \text{ if } x_{0}x_{e} = 01\\
    -1 & \text{ if } x_{0}x_{e} = 10\\
    0 & \text{otherwise,}
    \end{cases}\]
    and extending the definition to other values in the way determined by the cocycle equation. This cocycle is nontrivial (Proposition~\ref{proposition.non.trivial.sqgrp.cocycle}). To illustrate how cocycles can be useful, notice that a square of side length $6$ whose border is colored with the cycle $(012)^{8}$ (clockwise, for instance) cannot be filled into a locally admissible pattern $p$ on the whole square, since for any configuration $x$ containing this pattern on position $\boldsymbol{0}$, we would have 
    $0 = c(\boldsymbol{0},x)$ and $c(\boldsymbol{0},x) = 8$, which is impossible.
\end{example}
Finally, we would like to mention that the notion of cohomology defined here is distinct from other notions of cohomology on similar symbolic spaces (look for instance at \cite{zbMATH06767696}).

\section{Square group and square-decomposability}

{This section introduces some key tools for the study of the cohomology of homshifts. Section~\ref{sec:basicdefs} summarizes some basic facts about the square group and square-decomposability for which proofs and more information can be found in \cite{CGHO25}; the following subsections introduce new concepts and results that we use in the article.}

\subsection{Basic definitions}\label{sec:basicdefs}

Consider a graph $G$ and a vertex $a$ of $G$. We denote by \notationidx{$\mathcal{U}_G[a]$}{set of non-backtracking walks on $G$ beginning at $a$} the set of non-backtracking walks on $G$ which begin at $a$, and the fundamental group of $G$, $\pi_1(G)[a]\subset \mathcal{U}_G[a]$, is the subset of cycles. Note that $(\pi_1(G)[a],\star)$ is a group, {and all groups $\pi_1(G)[a]$, $a \in V_G$, are isomorphic}. Denote by \notationidx{$\Delta(G)[a]$}{group of the square-decomposable cycles of $G$ beginning at $a$} the smallest normal subgroup of \notationidx{$\pi_1(G)[a]$}{fundamental group of $G$ on base point $a$} which contains all cycles $p \star s \star p^{-1}$, where $s$ is a square such that $s_0 = p_{l(p)}$ and $p_0 = a$; {again all groups $\Delta(G)[a]$, $a \in V_G$, are isomorphic}.

\begin{definition}
    Two walks $p,q$ on $G$ which begin at $a$ and end at the same vertex \textbf{differ by a square} when $p \star q^{-1} = w \star s \star w^{-1}$ for some walk $w$ and some square $s$. They are said to be \textbf{square-equivalent} when $p \star q^{-1} \in \Delta(G)[a]$.
\end{definition}

\begin{definition}
    A cycle on $G$ is said to be \textbf{square-decomposable} when it is square-equivalent to a trivial cycle. 
    We say that $G$ is \textbf{square-decomposable} when all of its cycles are square-decomposable.
\end{definition}
Note that a square-decomposable cycle is of even length. In particular, a square-decomposable graph is bipartite.
\begin{definition}\label{def:sqgroup}
    Let us denote by \notationidx{$\pi_1^{\square}(G)[a]$}{square group of $G$ on base point $a$} the quotient of the group $\pi_1(G)[a]$ by $\Delta(G)[a]$ 
    and denote by \notationidx{$\mathcal{U}^{\square}_G[a]$}{square cover of $G$ on base point $a$} the quotient of $\mathcal{U}_G[a]$ by the square-equivalence relation.
\end{definition}

\begin{remark}
    All the groups $\pi_1^{\square}(G)[a]$, $a \in V_G$, are isomorphic. We call its isomorphic class the \textbf{square group} of $G$, and we denote it by $\pi_1^\square(G)$.
\end{remark}

\subsection{The even square group\label{section.even.square.group}}

We introduce the subgroup of the square group {that consists of cycles of even length. One way to see why this restriction is natural is that a partial homomorphism from the border of a region of $\mathbb Z^2$ corresponds to a cycle of even length in $G$. These are the relevant cycles with regards to cohomology, as we will see} in Section \ref{section.cocycle.infinite}. 

\begin{definition}
Let $\mathcal E_G[a] \subset \pi_1(G)[a]$ denote the set of even length cycles in $G$. 
Then $(\mathcal E_G[a],\star)$ is a subgroup of $(\pi_1(G)[a],\star)$.
Note that $\Delta(G)[a]$ is also a normal subgroup of $\mathcal E_G[a]$. Their quotient is denoted by \notationidx{$\mathcal{E}_G^{\square}[a]$}{even square group of $G$ on base point $a$}:
\[\mathcal{E}_G^{\square}[a] \coloneqq\mathcal E_G[a]/\Delta(G)[a].\]
The groups $\mathcal{E}_G^{\square}[a]$, $a \in G$ are isomorphic and their equivalence class is called the \textbf{even square group} of $G$ and is denoted by $\mathcal{E}^{\square}_G$. 
\end{definition}
 The following implies the second equivalence of Theorem~\ref{theorem: cohomological triviality}.
\begin{proposition}\label{Prop: basic_properties_square}
\textbf{1.} If $G$ is bipartite then 
${\mathcal E_G=\pi_1(G)}$ 
and hence 
${\mathcal{E}_G^{\square}=\pi_1^\square(G)}$. 
\textbf{2.} If $G$ is not bipartite then 
${ \pi_1^\square(G)/\mathcal{E}_G^{\square}\cong \Z/2\Z}$.
\end{proposition}
\begin{proof}
\textbf{1.} When $G$ is bipartite then all cycles have even length and hence $\mathcal E_G=\pi_1(G)$. 
\textbf{2.} By the third isomorphism theorem for groups, since $\mathcal{E}_G$ is normal in $\pi_1(G)$, it is sufficient to prove that $\pi_1(G)/\mathcal E_G\cong \Z/2\Z.$
In this proof, we denote by $\overline{c}$ the class of $c \in \pi_1(G)$ in $\pi_1(G)/\mathcal E_G$. For all $c,c' \in \pi_1(G)$ such that $\overline{c} \neq 1$ and $\overline{c}'\neq 1$, $c$ and $c'$ have odd length. Thus $c \star c'$ has even length and this implies that $\overline{c} \star \overline{c}' = 1$. As a consequence, $\pi_1(G)/\mathcal E_G$ is generated by one element whose square is $1$, meaning that $\pi_1(G)/\mathcal E_G \cong \Z/2\Z$.\qedhere
\end{proof}

\begin{proposition}\label{proposition: adding self loops}
    Let $G$ be a bipartite graph, and $G'$ obtained from $G$ by adding a self-loop. Then $\pi^\square_1(G') = \pi^\square_1(G) \ast \Z/2\Z$ (free product). In particular, $\mathcal{E}_{G}^{\square}$ is trivial if and only if $\mathcal{E}_{G'}^{\square}$
    is trivial.
\end{proposition}
\begin{proof}
{Denote by $a$ the vertex of $G$ on which the self-loop is added to obtain $G'$.}
Denote by $e\coloneqq aa$ the cycle corresponding to the self loop on the vertex $a$. Notice that $e$ doesn't appear in any square in $G'$, as $G$ is bipartite. A presentation of the group $\pi_1(G')[a]$ can thus be obtained from any presentation of $\pi_1(G)[a]$ by adding a generator corresponding to the cycle $e$ and the relation $e^2 = 1$ (since $e^2$ is a backtrack). This means that $\pi^\square_1(G')[a] = \pi^\square_1(G)[a] \ast \Z/2\Z$. Thus, if $ \mathcal{E}_{G}^{\square}[a] = \pi^\square_1(G)[a]$ is trivial, then $\pi^\square_1(G')[a] = \Z/2\Z$. {By Proposition \ref{Prop: basic_properties_square}, point 2, since $G'$ is not bipartite,} $\mathcal{E}_{G'}^{\square}[a]$ is trivial. Conversely, any nontrivial element in $\mathcal{E}_{G}^{\square}[a]$ is nontrivial in $\pi^\square_1(G')[a]$. Since it is in $\mathcal{E}_{G'}^{\square}[a]$, $\mathcal{E}_{G'}^{\square}[a]$ is nontrivial.
\end{proof}

To end this section, let us prove that the even square group of any graph is the square group of some bipartite graph. Provided a non-bipartite graph $G$, its \textbf{bipartite cover} is the graph $G'$ such that $V_{G'}= V_G\times\{0,1\}$ and 
\[E_{G'} =\{((v,i),(w,1-i))~:~(v,w)\in E_G\text{ and }i\in \{0,1\}\}.\]
It is easy to check that $G'$ is a bipartite connected graph where the partite classes are $V_G\times \{0\}$ and $V_G\times \{1\}$.

\begin{proposition}\label{proposition: bipartite graphs and stuff}
For $G$ a non-bipartite graph and $G'$ its bipartite cover, we have:
\[\mathcal E^\square_{G}\cong\mathcal E^\square_{G'}=\pi_1^\square(G').\]
\end{proposition}
\begin{proof}
Let $\phi:G'\to G$ be the (covering) map such that for all $(v,i) \in V$, $\phi((v,i))=v$. Let us fix a vertex $a' \in G'$. Since $\phi$ is a graph homomorphism, the map from $\mathcal E_{G'}[a']$ to $\mathcal E_{G}[\phi(a')]$ given by applying $\phi$ vertex by vertex is a group isomorphism. {This comes directly from the walk-lifting property of covering maps, which states that any walk in the image graph has a unique preimage by the pointwise application of this covering map, up to the choice of initial vertex \cite[Proposition \textbf{3.9}]{CGHO25}.}

Furthermore the restriction of this map to $\Delta(G')[a']$ is an isomorphism onto $\Delta(G)[\phi(a')]$. 
It follows that the $\mathcal{E}^{\square}_{G'}$ and $\mathcal{E}^{\square}_{G}$ are isomorphic.
\end{proof}

\begin{proposition}\label{proposition.undecidability}
    It is undecidable whether the even square group of a graph $G$ corresponding to a mixing two-dimensional homshift is trivial or not.
\end{proposition}

\begin{proof}
    We know that every finitely presented group $H$ is the square group of a bipartite graph \cite[Theorem \textbf{5.1}]{CGHO25}. As a consequence of Proposition \ref{proposition: adding self loops}, $H$ is the even square group of a graph whose two-dimensional homshift is mixing. This yields the statement, 
    as it is undecidable whether a finitely presented group is trivial or not.
\end{proof}

\subsection{Characterization of square-decomposability\label{section.technical.results}}

In this section, we characterize square-decomposability of walks in terms of pattern-completion (Lemma \ref{lemma.box}) which will be useful in the remainder of the article. 

While the notions of walks, differing by a square and square-decomposability have been introduced for finite graphs, they can easily be extended to infinite graphs. We will thereby not repeat the definitions. For most of this section we will be recalling some past results which will be useful in the later sections.

\begin{notation}
    For a configuration $x$ of a $d$-dimensional homshift $X^d_G$ and $w$ a walk on $\mathbb{Z}^d$ we denote by \notationidx{$x_w$}{walk in a configuration $x$ traced by a walk $w$ on $\Z^d$} the 
walk $x_{w_0} \ldots x_{w_{l(w)}}$ on $G$. 
\end{notation}

{For every $r,r' \ge 0$, we denote by $\partial \boldsymbol{B}^2(r,r',s,s')$ the set 
\[\left(\llbracket r , s \rrbracket \times \llbracket r' , s' \rrbracket\right) \backslash \left(\llbracket r+1 , s-1 \rrbracket \times \llbracket r'+1 , s'-1 \rrbracket\right).\]
In particular $\partial \boldsymbol{B}^2 (n) = \partial \boldsymbol{B}^2 (n,n,n,n)$ for all $n$.}

\begin{notation}\label{notation.clockwise}
{For all $r,r',s,s'$, we identify any locally admissible pattern $p$ on $\partial \boldsymbol{B}^2(r,r',s,s')$ with the cycle $p_{l} \odot p_{u} \odot p_{r} \odot p_{d}$ obtained by reading it clockwise: \notationidx{$p_{l}$}{walk corresponding to the left side of a rectangular pattern $p$}  = $p_{(r,r')} \ldots p_{(r,s')}$, \notationidx{$p_{u}$}{walk corresponding to the upper side of a rectangular pattern $p$} = $p_{(r,s')} \ldots p_{(s,s')}$, \notationidx{$p_{r}$}{walk corresponding to the right side of a rectangular pattern $p$} = $p_{(s,s')} \ldots p_{(s,r')}$ and \notationidx{$p_{d}$}{walk corresponding to the down side of a rectangular pattern $p$} = $p_{(s,r')} \ldots p_{(r,r')}$.}
\end{notation}

The following is straightforward:

\begin{lemma}\label{lemma.varphi.sqdiff}
Let us consider a configuration $x$ of a homshift $X_G^2$ and $w,w'$ two walks on $\mathbb{Z}^2$ which differ by a square. 
Then $\varphi(x_w)$ and $\varphi(x_{w'})$ are equal or differ by a square.
\end{lemma}
Every cycle in $\Z^2$ is square-decomposable. Thus repeated application of Lemma~\ref{lemma.varphi.sqdiff} gives us the following.
\begin{lemma}\label{lemma.border.sqdec}
    Consider an {admissible}
    pattern $p$ on $\boldsymbol{B}^2(r)$ for some $r \ge 0$. The cycle $\varphi(p_{|\partial \boldsymbol{B}^2(r)})$ is square-decomposable.
\end{lemma}

An easy adaptation of the proof of~\cite[Lemma \textbf{6.3}]{gangloff2022short} implies the result below. 
{We say that two walks $p$ and $q$ such that $l(p) = l(q)$ are \emph{neighbors} when $p_i$ is a neighbor of $q_i$ for all $i$.} Notice that the definitions of differing by a square differ between our articles, but they are equivalent as shown in \cite[Lemma \textbf{4.17}]{CGHO25}.

\begin{lemma}\label{lemma.intermediate.box}
Let $c,c'$ be two non-backtracking cycles which differ by a square, $t$ a backtrack which starts at $c_0$, and $k \geq \max\{0, \frac{l(c')-l(c)}2\}$ an integer. Then
there exists a sequence of cycles $c^{(0)}, \ldots , c^{(m)}$ such that 
\[
c^{(0)} = t^k \odot c,\quad  c^{(m)} = t^{k+(l(c)-l(c'))/2} \odot c',\]
and for all $i$, $c^{(i)}$ and $c^{(i+1)}$ are neighbors and 
$c^{(i)}$ begins and ends with $t_0$ (resp. $t_1$) when $i$ is even (resp. odd). 
\end{lemma}
Note that, to apply~\cite[Lemma \textbf{6.3}]{gangloff2022short} in the case $l(c') < l(c)$, we reverse the parameters and the order of the resulting sequence.
We end the section with an equivalent condition for square-decomposability which will be crucial in the proof of Theorem \ref{theorem: cohomological triviality}.
\begin{lemma}\label{lemma.box}
    A non-backtracking cycle $\gamma$ on $G$ is square-decomposable if and only if $l(\gamma)\in 2\N$ and there exist $k,n \in \mathbb{N}$, $t$ a backtrack which starts at $\gamma_0$ and a {locally admissible} pattern $R$ with support $\llbracket 0 , l(\gamma) + 4k \rrbracket \times \llbracket 0 , 2n\rrbracket$ such that: \[{R_d = t^{-k} \odot \gamma^{-1} \odot t^{-k} \quad R_u = t^{l(\gamma)/2 + 2k} \quad R_l = t^n \quad R_r = t^{-n}}\] 
\end{lemma}

\begin{remark}\label{remark:any_t}
    If the right-hand side of this equivalence holds for some $\gamma, k,n$ and $t$, we may replace $k,n$ with any larger integers and $t$ with any cycle of length 2 on $\gamma_0$. 
\end{remark}

\begin{proof}
    $(\Rightarrow)$ Since $\gamma$ is square-decomposable, there exists a sequence of cycles $(\gamma^{(i)})_{i=0\ldots m}$ such that for all $i$, $\gamma^{(i)}$ and $\gamma^{(i+1)}$ differ by a square, $\gamma = \gamma^{(0)}$ and $l(\gamma^{(m)}) = 0$. 
    Since two cycles differing by a square have the same parity, $l(\gamma^{(i)})\in2\N$.
    Pick a backtrack $t$ starting at $\gamma_0$. Fix any $k\geq\max\{0,\frac{\max_i l(\gamma^{(i)})-l(\gamma^{(0)})}2\}$.
    
    For every $i$, let $W_i\coloneqq
		t^{a_i}\odot\gamma^{(i)}\odot t^k$ where $a_i= k+\frac{l(\gamma^{(0)})-l(\gamma^{(i)})}{2}\geq 0$ is chosen so that $l(W_i)$ is constant. Note that $a_i \geq \frac{\max_j l(\gamma^{(j)})-l(\gamma^{(i)})}2 \geq \frac{l(\gamma^{(i+1)})-l(\gamma^{(i)})}2$.
        Therefore we can apply Lemma~\ref{lemma.intermediate.box} with parameter $a_i$ and add $t^k$ (resp. $\omega(t)^k$) to the right of each even (resp. odd) cycle in the resulting sequence, giving a sequence of neighboring cycles from $W_i$ to $W_{i+1}$. The sequence has even length, say $2n_i$, because its first and last cycles begin and end at $t_0$. It therefore defines a rectangular pattern $R^{(i)}$ satisfying
	\[
		R^{(i)}_d=W_i^{-1},\qquad
		R^{(i)}_u=W_{i+1},\qquad
		R^{(i)}_l=t^{n_i},\qquad
		R^{(i)}_r=t^{-n_i}.
	\]
	These rectangles can be stacked because
		$R^{(i)}_u=W_{i+1}=(R^{(i+1)}_d)^{-1}$, the lower boundary being read
		from right to left. Set $n=\sum_{i=0}^{m-1}n_i$ and let $R$ be the
		resulting pattern.
        By definition $W_m
		=t^{k+l(\gamma)/2}\odot t^k
		=t^{2k+l(\gamma)/2}$ completing the argument.

    $(\Leftarrow)$ 
    {By Lemma \ref{lemma.border.sqdec}, we  have that $\varphi(\partial R)\in \Delta(G)[\gamma_0]$. Since $\varphi(R_l\odot R_u\odot R_r) = \varphi(t^n \odot t^{l(\gamma)/2+2k} \odot t^{-n})$ which is trivial, 
we have that $\varphi(\gamma)\in \Delta(G)[\gamma_0]$. This completes the proof.}
\end{proof}   

\section{The square group cocycle\label{section.sq.cocycle}}

 In this section, for every $G$, we define a cocycle with values in $\pi_1^{\square}(G)$ that we call square group cocycle (Section \ref{section.definition.sq.cocycle}). Using this cocycle, we prove direction $(\Rightarrow)$ in the statement of Theorem \ref{theorem: cohomological triviality} (Section \ref{section.non.trivial.square.group}).

\subsection{Definition\label{section.definition.sq.cocycle}}

\begin{notation}
    For a graph $G$, a spanning tree $T$ of $G$ and $a,a' \in G$, we denote by \notationidx{$p_T^a(a')$}{unique simple walk from $a$ to $a'$ in a spanning tree $T$ of a graph} the unique simple walk on $T$ from $a$ to $a'$.
\end{notation}

\begin{notation}Given a graph $G$ and any walk $w$ on $G$, we denote by \notationidx{$p_G^{\square}(w)$}{equivalence class of a walk $w$ for the square-equivalence relation} the equivalence class of $w$ for the square-equivalence relation.\end{notation}

For $x\in X^d_G$ and walk $p$ from $\boldsymbol{0}$ to $\boldsymbol{n}$ in $\Z^d$, $x_p$ is a walk on the graph. Note that all choices of walks $p$ from $\boldsymbol{0}$ to $\boldsymbol{n}$ lead to square equivalent walks $x_p$. Furthermore, they satisfy the following natural cocyclic property: given a walk $p$ from $\boldsymbol{0}$ to $\boldsymbol{n}$ and $q$ from $\boldsymbol{n}$ to $\boldsymbol{n+m}$, we have that $x_p\odot x_q=x_{p\odot q}$. This leads to a natural definition of a cocycle on the homshift $X^d_G$.
\begin{definition}\label{def.sqgrp.cocycle}
We denote by \notationidx{$c_G^{T,a}$}{square-group cocycle of $X^{d}_G$ corresponding to the rooted spanning tree $(T,a)$ of the graph $G$} the map $\mathbb{Z}^{d} \times X_G^{d}  \rightarrow \pi_1^{\square}(G)[a]$ 
defined as follows.
For all $\boldsymbol{n} \in \mathbb{Z}^d$, $x \in X^d_G$, and any walk $p$ in $\mathbb{Z}^d$ from $\boldsymbol{0}$ to $\boldsymbol{n}$,
\[c_G^{T,a}(\boldsymbol{n},x) \coloneqq \left(p_G^{\square}(p_T^a(x_0) \star x_p \star (p_T^a(x_{\boldsymbol{n}}))^{-1})\right)^{-1}.\]
This is well-defined, as all walks from $\boldsymbol{0}$ to $\boldsymbol{n}$ are square-equivalent and $p_G^{\square}$ is by definition constant on a square-equivalence class.
\end{definition}
Note that square-equivalence classes of the walks $x_p$ are elements of the square cover, which is not a group. However, as mentioned in \cite{CGHO25}, the square cover can be decomposed into copies of a square-decomposable graph on which the square group acts by permuting these copies, and the square cocycle keeps track of which copy the current walk belongs to.

\begin{remark}
Some readers may find the inverse unnatural in this definition. This is necessary, due to the order of the product in the cocycle equation \eqref{equation.product.along.walk}, as we follow K.~Schmidt's definition of a cocycle \cite{Schmidt95}. We will soon prove that $c_G^{T,a}$ indeed defines a cocycle.
\end{remark}

\begin{lemma}
The map $c_G^{T,a}$ is a continuous cocycle on $X_G^d$ with values in the square group of $G$. We call such a cocycle \textbf{the square-group cocycle}.
\end{lemma}

\begin{proof}
     Continuity is straightforward from the definition. Let us prove that $c_G^{T,a}$ is a cocycle. For this, it is sufficient to see that for all $\boldsymbol{m},\boldsymbol{n} \in \Z^d$ and all $x \in X_G^d$, 
    \[\left(c_G^{T,a}(\boldsymbol{n}+\boldsymbol{m},x)\right)^{-1}=\left(c_G^{T,a}(\boldsymbol{n},x)\right)^{-1} \star \left(c_G^{T,a}( \boldsymbol{m},\sigma^{\boldsymbol{n}}(x))\right)^{-1}.\]
    The right-hand side of this equation is equal to: 
    \[p_G^{\square}(p_T^a(x_{\boldsymbol{0}}) \star x_p \star (p_T^a(x_{\boldsymbol{n}}))^{-1}) \star p_G^{\square}(p_T^a(x_{\boldsymbol{n}}) \star x_q \star (p_T^a(x_{\boldsymbol{n}+\boldsymbol{m}}))^{-1})\]
     where $p$ is a walk in $\mathbb{Z}^d$ from $\boldsymbol{0}$ to $\boldsymbol{n}$, and $q$ is a walk in $\mathbb{Z}^d$ from $\boldsymbol{n}$ to $\boldsymbol{n+m}$ (this is equivalent to $
x_q=
(\sigma^{\boldsymbol{n}}(x))_{\hat{q}}$ for some walk $\hat{q}$ in $\mathbb{Z}^d$ from $\boldsymbol{0}$ to $\boldsymbol{m}$).
    This can be simplified into: 
    \[\begin{array}{r@{\,=\,}l}
         p_G^{\square}(p_T^a(x_{\boldsymbol{0}}) \star x_p \star x_q \star (p_T^a(x_{\boldsymbol{n}+\boldsymbol{m}}))^{-1}) &
     p_G^{\square}(p_T^a(x_{\boldsymbol{0}}) \star {(x_{p} \odot x_{q})} \star (p_T^a(x_{\boldsymbol{n}+\boldsymbol{m}}))^{-1}) \\
    & p_G^{\square}(p_T^a(x_{\boldsymbol{0}}) \star x_{p \odot q} \star (p_T^a(x_{\boldsymbol{n}+\boldsymbol{m}}))^{-1}),
    \end{array}\]
    which corresponds to the definition of the left-hand side.
\end{proof}

Note that the cocycle $c_G^{T,a}$ depends on the choice of $T$ and $a$. However, different choices yield continuously cohomologous cocycles:

\begin{proposition}
    All square group cocycles of $X^d_G$ are continuously cohomologous.
\end{proposition}

\begin{proof}
The cocycles $c_G^{T,a}$ and $c_G^{T,a'}$, for two choices of base points $a,a'\in G$, are the same cocycle up to isomorphism: the morphism $\beta : \pi_1(G)[a] \to \pi_1(G)[a']$ given by $c \mapsto p_T^{a'}(a)\star c \star p_T^a(a')$ is an isomorphism and square-equivalent cycles have square-equivalent images, so it factorizes onto a morphism $\overline\beta : \pi_1^\square(G)[a] \to \pi_1^\square(G)[a']$ and it is straightforward to check that $c_G^{T,a'} = \overline \beta\circ c_G^{T,a}$.

Now consider cocycles $c_G^{T,a}$ and $c_G^{T',a}$ for two choices $T$ and $T'$ of spanning trees. Define a transfer function $b_T^{T'} : X^d_G \to \pi^\square_1 (G)[a]$ by \[b_T^{T'} : x \mapsto p_G^\square\left(p_{T}^a(x_0) \star (p_{T'}^a(x_0))^{-1}\right)\]
which is clearly continuous as it depends only on $x_0$. For any $x\in X^d_G$ and $\boldsymbol{n}\in \Z^d$, we have:
\begin{align*}
&b_T^{T'}(\sigma^{\boldsymbol{n}}(x))^{-1} \star c_G^{T,a}(\boldsymbol{n},x) \star b_T^{T'}(x)\\
&=p_G^\square\left( p_{T'}^a(x_{\boldsymbol{n}})\star p_T^a(x_{\boldsymbol{n}})^{-1} \star p_T^a(x_{\boldsymbol{n}}) \star x_p^{-1} \star p_T^a(x_0)^{-1}  \star p_T^a(x_0) \star (p_{T'}^a(x_0))^{-1}\right) 
\\
& =p_G^\square\left( p_{T'}^a(x_{\boldsymbol{n}})\star x_p^{-1} \star (p_{T'}^a(x_0))^{-1}\right)\\
& =  c_G^{T',a}(\boldsymbol{n},x)
\end{align*}
We have proved that $c_G^{T,a}$ and $c_G^{T',a}$ are continuously cohomologous.
\end{proof}

\subsection[Proof of Theorem 1.1 \texorpdfstring{$ (\Rightarrow)$}{(⇒)}]{Proof of Theorem \ref{theorem: cohomological triviality} ($\Rightarrow$)\label{section.non.trivial.square.group}}

\begin{definition}
    For all $t_0,t_1$ neighbors in $G$, we call $d$-dimensional $(t_0,t_1)$-\textbf{chessboard configuration} the configuration with value $t_0$ over $\{\boldsymbol{n} \in \mathbb{Z}^d : |\boldsymbol{n}|_1 \in 2 \mathbb{Z}\}$ and $t_1$ over $\{\boldsymbol{n} \in \mathbb{Z}^d : |\boldsymbol{n}|_1 \in 2 \mathbb{Z}+1\}$.
\end{definition}

\begin{proposition}\label{proposition.non.trivial.sqgrp.cocycle}
   If the even square group of $G$ is nontrivial then none of the square-group cocycles of $X^d_G$ are trivial.
\end{proposition}

\begin{proof}
Fix a spanning tree $T$ of $G$, $a \in G$, and $\gamma$ a non-backtracking cycle on $G$ of even length beginning at $a$ such that $p_G^{\square}({\gamma}) \neq 1_{\pi_1^{\square}(G)}$.
Denote by $x$
the $d$-dimensional {$(\gamma_0,\gamma_1)$-chessboard configuration.
Consider the pattern $p$ on $\llbracket 0, l(\gamma)\llbracket ^2$ given by
        \[\begin{array}{cccc} \gamma_{l(\gamma) - 1} & \gamma_ 0 & \cdots & \gamma_{l(\gamma)-2}\\\vdots & \vdots & \ddots & \vdots \\
    \gamma_1 & \gamma_2 & \ldots & \gamma_{0}\\
    \gamma_0 & \gamma_1 & \ldots & \gamma_{l(\gamma)-1}\end{array}.\]
    Let $y'\in X^2_G$ be the configuration given by $y'_{\boldsymbol u + \llbracket 0, l(\gamma)\llbracket ^2}=p$ for all $\boldsymbol{u}\in l(\gamma)\Z^2$. For $d=2$, set $y \coloneqq y'$. For $d>2$, let $y\in X^d_G$ be given by 
    \[y_{\Z^2\times \{\boldsymbol{v}\}}=\begin{cases}
        y'&\text{ if }|\boldsymbol{v}|_1 \text{ is even}\\
        \sigma^{\boldsymbol{e}_1}(y')&\text{ if }|\boldsymbol{v}|_1 \text{ is odd}.
        \end{cases}\]

    }
    Let us assume for the sake of contradiction that $c_G^{T,a}$ is trivial, meaning that there exists a group homomorphism $\theta$ and a continuous function $b : X^d_G \rightarrow \pi_1^{\square}(G)[a]$ such that for all $\boldsymbol{n} \in \mathbb{Z}^d$ and $x \in X^d_G$, 
    $c_G^{T,a}(\boldsymbol{n},x) = b(\sigma^{\boldsymbol{n}}(x))^{-1} \theta(\boldsymbol{n}) b(x)$. 
    Taking $\boldsymbol{n} = l(\gamma)\boldsymbol{e}^1$, 
    we have:
    \[c_G^{T,a}(l(\gamma)\boldsymbol{e}^1,x) = b(\sigma^{l(\gamma)\boldsymbol{e}^1}(x))^{-1} \theta(l(\gamma)\boldsymbol{e}^1) b(x).\]

    Since $l(\gamma)$ is even, we have that $\sigma^{l(\gamma)\boldsymbol{e}^1}(x) = x$ and thus $b(\sigma^{l(\gamma)\boldsymbol{e}^1}(x)) = b(x)$. Furthermore, since the cycle $\varphi (x_{0\boldsymbol{e}^1} x_{1\boldsymbol{e}^1} \ldots x_{l(\gamma)\boldsymbol{e}^1})$ is trivial, by definition of $c_G^{T,a}$, we get $c_G^{T,a}(l(\gamma)\boldsymbol{e}^1,x) = 1_{\pi_1^\square(G)}$. We thus get $\theta(l(\gamma)\boldsymbol{e}^1) = 1_{\pi_1^\square(G)}$.
     
    The same reasoning for $y$ instead of $x$, using the fact that $\sigma^{l(\gamma)\boldsymbol{e^1}}(y) = y$, leads to:
    \begin{align*}
    (p_G^{\square}(\gamma))^{-1} = c_G^{T,a}(l(\gamma)\boldsymbol{e}^1,y)& =
    b(\sigma^{l(\gamma)\boldsymbol{e}^1}(y))^{-1} \theta(l(\gamma)\boldsymbol{e}^1) b(y)\\
    &=b(y)^{-1}\theta(l(\gamma)\boldsymbol{e}^1)b(y).\end{align*}
    We have seen that $\theta(l(\gamma)\boldsymbol{e}^1) = 1_{\pi_1^\square(G)}$, hence $p_G^{\square}(\gamma) = 1_{\pi_1^\square(G)}$.
    This contradiction concludes the proof.
\end{proof}

\section[Proof of Theorem 1.1 \texorpdfstring{$ (\Leftarrow)$}{(<=)}, dimension two]{Proof of Theorem \ref{theorem: cohomological triviality} ($\Leftarrow$), dimension two\label{section.cocycle.infinite}}

In this section, we prove direction $(\Leftarrow)$ 
in the statement of Theorem \ref{theorem: cohomological triviality} (Section \ref{section.proof.equivalence}), after introducing the notions of Gibbs equivalence and strip-gluing property (Section \ref{section.specification}).

\subsection{Gibbs equivalence and strip-gluing property\label{section.specification}}

For a two-dimensional subshift $X$, we denote by \notationidx{$\Delta_X$}{Gibbs equivalence relation of a subshift $X$} the subset of $X \times X$ defined by $(x,x') \in \Delta_X$ when $x$ and $x'$ differ only on finitely many positions in $\mathbb{Z}^2$. We also denote by \notationidx{$\Delta_X(x)$}{equivalence class of $x$ for the relation $\Delta_X$} the equivalence class of $x$ for this relation. The relation $\Delta_X$ is referred to as \textbf{Gibbs equivalence}.

\subsubsection*{Strip-gluing property}

 The following definition is inspired by the specification property (Definition \textbf{2.2}) in \cite{Schmidt95}.

\begin{definition}\label{def:stripgluing}
    We say that $X$ has the \textbf{strip-gluing property} relative to $x'' \in X$ when for all $r \ge 0$, $x,x' \in \Delta_X(x'')$ which coincide on $\boldsymbol{B}^2(r)$, there exists some $y \in \Delta_X(x'')$ which coincides with $x$ on $\llbracket -r , {+\infty}\llbracket \times \llbracket -r , r \rrbracket $ and with $x'$ on $\rrbracket {-\infty} , r\rrbracket \times \llbracket -r , r \rrbracket$.
\end{definition}

\begin{remark}
    In the previous definition, it would appear natural for the property to apply to strips in any direction and to relax the requirement to hold only for $r$ large enough; in the case of homshifts this makes no difference.
\end{remark}

\begin{lemma}\label{lemma.specification}
    If the even square group of $G$ is trivial, $X_G^2$ has the strip-gluing property relative to the $(t_0,t_1)$-chessboard configuration for any two neighbors $t_0,t_1$ in $G$.  
\end{lemma}

\begin{proof}
    Let $x''$ be such a chessboard configuration and fix any $r \ge 0$ and $x,x' \in \Delta_X(x'')$ which coincide on $\boldsymbol{B}^2(r)$. The infinite pattern $p$ on $\mathbb{Z} \times \llbracket -r, r \rrbracket$ which coincides with $x$ on $\llbracket -r, +\infty \llbracket \times \llbracket -r, r \rrbracket$ and with $x'$ on $\rrbracket {-\infty} , r \rrbracket \times \llbracket -r, r \rrbracket$ is locally admissible. We choose $l > r$ large enough so that $x,x'$ coincide with $x''$ outside of $\boldsymbol{B}^2(l)$ and such that $l-r$ is even. In particular, $p$ coincides with $x''$ outside of $\llbracket {-l},l \rrbracket \times \llbracket {-r} , r\rrbracket$. We extend $p$ so that it coincides with $x''$ everywhere outside of $\boldsymbol{B}^2(l)$. Notice that $p$ is now defined everywhere except on $(\rrbracket {-l},l \llbracket \times \rrbracket {-l} , {-r}\llbracket) \cup (\rrbracket {-l},l \llbracket \times \rrbracket r , l\llbracket)$. Furthermore, the restriction of $p$ to $\partial (\rrbracket {-l},l \llbracket \times \rrbracket {-l} , {-r}\llbracket)$ can be written as $t^{\frac {l-r}2} \odot t^l \odot t^{\frac {l-r}2} \odot (t^k \odot \gamma \odot t^k)$, following Notation~\ref{notation.clockwise}, for some $k$ and some cycle $\gamma$, and similarly for $\partial(\llbracket -l,l \rrbracket\times \{-r\})$. 
    By choosing $l$ large enough, we apply Lemma~\ref{lemma.box} to extend $p$ into a configuration of $\Delta_X(x'')$.
\end{proof}

\begin{remark}\label{remark.specification}
    In this article, we characterize homshifts having trivial cohomology in discrete groups using the strip-gluing property (Definition \ref{def:stripgluing}). In \cite{Schmidt95}, Schmidt studied cohomological triviality in more general groups, namely locally compact second countable ones, by using {a very similar property called} specification. Both properties consist in being able to glue together locally admissible patterns taken from  configurations in the same Gibbs equivalence class, but strip-gluing applies to strip patterns while specification applies to cone patterns.

    We cannot use the specification property for our classification because of the following example, which has trivial cohomology but does not satisfy Schmidt's specification property. Consider the Kenkatabami graph in Figure \ref{figure.kenkatabami}. Let $G$ be the graph obtained by adding a self-loop on the central vertex $\omega$. Since the Kenkatabami graph has trivial square group, $G$ has a trivial even square group (Proposition \ref{proposition: adding self loops}), so it has trivial cohomology. We describe two configurations $x,x'$ in the Gibbs equivalence class of an arbitrary chessboard configuration, illustrated in Figure~\ref{fig:spec.kenkatabami}: given $n\in \N$, $x$ has the word $c^{-n}$ written along the top border of the left cone and $x'$ has the word $c^n$ written along the top border of the right cone. If we assume the specification property, there is some configuration $y$ which is equal to $x$ on the left cone and equal to $x'$ in the right cone, which forces in each case a triangular pattern where words $c^{\alpha n}$ and $c^{- \alpha n}$ appear, where $\alpha$ is the sine of the cone slope. However, we know that the minimal possible distance between two occurrences of these words is $\Theta(\log(n))$ \cite[Theorem \textbf{6.15}]{gangloff2022short}, so $y$ cannot exist.

\begin{figure}
\begin{center}
\begin{tikzpicture}[scale=0.4]
\draw (-2,0) -- (0,0) -- (2,0);
\draw (-2,0) -- (-3,1.72) -- (-1,1.72);
\draw (2,0) -- (3,1.72) -- (1,1.72);
\draw (-1,1.72) -- (0,0) -- (1,1.72);
\draw (-1,-1.72) -- (0,0) -- (1,-1.72);
\draw (-1,-1.72) -- (0,-3.44) -- (1,-1.72);
\draw (-1,1.72) -- (0,3.44) -- (1,1.72);
\draw (-3,1.72) -- (0,5.88) -- (0,3.44);
\node at (0,6.7) {$\epsilon_1$};
\draw[fill=gray!90] (2,0) circle (3pt);
\node at (1.5,0.5) {$\mu_6$};
\draw[fill=gray!90] (-2,0) circle (3pt);
\node at (-1.5,0.5) {$\mu_3$};
\draw (0,5.88) -- (3,1.72);
\draw (-2,0) -- (-3,-1.72) -- (-1,-1.72);
\draw (2,0) -- (3,-1.72) -- (1,-1.72);
\draw[fill=gray!90] (1,-1.72) circle (3pt);
\node at (1.75,-2.5) {$\mu_5$};
\draw[fill=gray!90] (-1,-1.72) circle (3pt);
\node at (-1.75,-2.5) {$\mu_4$};
\draw[fill=gray!90] (3,-1.72) circle (3pt);
\node at (3.5,-1) {$\delta_3$};
\draw[fill=gray!90] (-3,-1.72) circle (3pt);
\node at (-3.5,-1) {$\delta_2$};

\draw[fill=gray!90] (1,1.72) circle (3pt);
\node at (1.5,2.25) {$\mu_1$};
\draw[fill=gray!90] (-1,1.72) circle (3pt);
\node at (-1.5,2.25) {$\mu_2$};
\draw[fill=gray!90] (3,1.72) circle (3pt);
\node at (3.65,2) {$\gamma_3$};
\draw[fill=gray!90] (-3,1.72) circle (3pt);
\node at (-3.65,2) {$\gamma_1$};

\node at (6,-3.75) {$\epsilon_3$};
\node at (-6,-3.75) {$\epsilon_2$};
\draw[fill=gray!90] (-5.225,-2.945) circle (3pt);
\draw[fill=gray!90] (5.225,-2.945) circle (3pt);
\draw (3,-1.72) -- (5.225,-2.945) -- (0,-3.44);
\node at (0,-4.6) {$\gamma_2$};
\draw[fill=gray!90] (0,-3.44) circle (3pt);
\draw[fill=gray!90] (0,5.88) circle (3pt);
\draw[fill=gray!90] (0,0) circle (3pt);
\node at (-1,-0.5) {$\omega$};
\draw[fill=gray!90] (0,3.44) circle (3pt);
\node at (-0.75,3.7) {$\delta_1$};
\draw (5.225,-2.945) -- (3,1.72);

\draw (-3,-1.72) -- (-5.225,-2.945) -- (0,-3.44);
\draw (-5.225,-2.945) -- (-3,1.72);
\end{tikzpicture}
\end{center}
\caption{The Kenkatabami graph.\label{figure.kenkatabami}}
\end{figure}
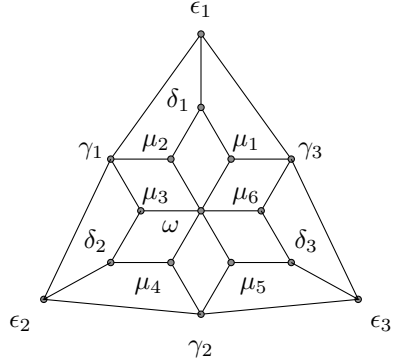

\begin{figure}
\begin{center}
\input{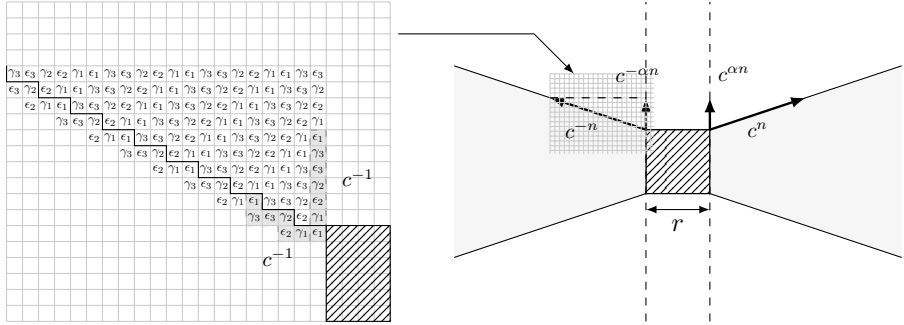}
\end{center}
\caption{Illustration for Remark \ref{remark.specification}. {The picture on the left illustrates how vertical $c^{-\alpha n}$ is forced by $c^{-n}$ written on the negative cone. We assume here for simplicity that the box's corners are included in the border of the cones.}}
\label{fig:spec.kenkatabami}
\end{figure}
\end{remark}

\subsubsection*{Density of Gibbs equivalence classes}

The following result is a consequence of \cite[Proposition \textbf{3.1}]{MR3743365} and the techniques introduced for the proof of \cite[Proposition \textbf{9.2}]{zbMATH07433654}. The details are left to the reader.

\begin{proposition}\label{proposition.equivalence.cover}Let $G$ be a finite graph. Fix an edge $(t_0, t_1)$ in $G$ and let $x$ be a $(t_0, t_1)$-chessboard configuration.
If $G$ is not bipartite then $\Delta_{X^d_G}(x)$ is dense.
\end{proposition}

\subsection{When the even square group is trivial\label{section.proof.equivalence}}
{To complete the proof of Theorem \ref{theorem: cohomological triviality}}
we are left to prove that when the even square group of $G$ is trivial, all the cocycles on $X^2_G$ are trivial. 
This will be achieved with
Proposition~\ref{theorem.finite.case}. The proof relies on variations of Proposition \textbf{3.1} and Theorem \textbf{3.2} in~\cite{Schmidt95}. 

\begin{definition}
    Consider a two-dimensional subshift $X$ and a continuous function $f : X \rightarrow \mathbb{G}$, where $\mathbb{G}$ is a discrete group.
    Since $f$ is continuous, for all $(x,x') \in \Delta_X$, there exists a minimum positive integer $m(x,x')$ such that $f(\sigma^{k\boldsymbol{e}^1} (x)) = f(\sigma^{k\boldsymbol{e}^1} (x'))$ for all integers $k$ such that $|k| > m(x,x')$. 
    We define $c_f^{\pm} : \Delta_X \rightarrow \mathbb{G}$ as follows for $(x,x') \in \Delta_X$ (the product notation corresponds to 
    $\prod_{k=0}^m a_k=a_0a_1\ldots a_m$): 
    \[c_f^{+}(x,x') \coloneqq \left(\prod_{k=0}^{m(x,x')} f\left(\sigma^{ {k}\boldsymbol{e}^1}(x)\right)^{{-1}}\right) \cdot \left(\prod_{k=0}^{m(x,x')} f\left(\sigma^{k \boldsymbol{e}^1}(x')\right)^{{-1}}\right)^{-1}\]
\[c_f^{-}(x,x') \coloneqq \left(\prod_{k=1}^{m(x,x')} f\left(\sigma^{-k \boldsymbol{e}^1}(x)\right)\right) \cdot \left(\prod_{k=1}^{m(x,x')} f\left(\sigma^{-k \boldsymbol{e}^1}(x')\right)\right)^{-1}\]
\end{definition} 
Note that the functions $c_f^{\pm}$ have the property that for all $(x,y), (y,z)\in \Delta_X$, 
\begin{equation}
    \label{equation:cocycle property}c_f^{\pm}(x,z)=c_f^{\pm}(x,y)c_f^{\pm}(y,z).
\end{equation}
This follows from the fact that, in the formula for $c_f^{\pm}$, $m(x, x')$ can be replaced with any larger integer.
While the definitions $c_f^+$ and $c_f^-$ might seem cumbersome, here is a point of view for an intuition underlying their definition. Let $c$ be a cocycle on a shift space $X$ and $f(x)=c(x,\boldsymbol{e}_1)$. Suppose the cocycle is cohomologous to a homomorphism with transfer function $b$. We then get 
\[c_f^+(x,x')=c_f^-(x,x')= (b(x))^{-1}\cdot b(x').\]

 Thus the functions $c_f^+$ and $c_f^-$ can help identify the transfer function. This also helps explain the hypothesis in the forthcoming proposition.

\begin{proposition}\label{prop.main}
    Let us consider a two-dimensional subshift $X$ and a continuous function $f : X \rightarrow \mathbb{G}$, where $\mathbb{G}$ is a discrete group. {Assume that there is some $x \in X$ such that $X = \overline{\Delta_X(x)}$ and that $X$ has the strip-gluing property relative to $x$.}
    If we have $c_f^{+} = c_f^{-}$
    then there exists a continuous function $b : X \rightarrow \mathbb{G}$ such that $y \mapsto b(\sigma^{\boldsymbol{e}^1}(y))^{-1} \cdot f(y) \cdot b(y)$ is constant on $X$.
\end{proposition}

\begin{proof}
Since $f$ is continuous, there exists $r^{*} \ge 0$ such that if $y,y'$ coincide on {$\boldsymbol B(r^*)$}, then $f(y) = f(y')$. 

\paragraph*{(i) When $(y,y') \in \Delta_X(x)$ agree on $B(r^*)$, $c_f^{+}(y,y') = 1_{\mathbb{G}}$.}
Indeed, let us consider two configurations $y,y' \in \Delta_X(x)$ which coincide on {$\boldsymbol B(r^*)$}. Since $X$ has the strip-gluing property relative to $x$, there exists $z \in \Delta_X(x)$ which coincides with $y$ on $\llbracket -r^* , +\infty\llbracket \times \llbracket -r^* , r^* \rrbracket $ and with $y'$ on $\llbracket - \infty, r^* \llbracket\times \llbracket - r^* , r^* \rrbracket$. Clearly we have $c_f^{-}(z,y') = 1_{\mathbb{G}}$ and $c_f^{+}(y,z) = 1_{\mathbb{G}}$. 
Since, by assumption, $c_f^{+} = c_f^{-}$, we also have $c_f^{+}(z,y') = 1_{\mathbb{G}}$, and therefore $c_f^{+} (y,y') = c_f^{+}(y,z) \cdot  c_f^{+}(z,y') = 1_{\mathbb{G}}$.
\paragraph*{(ii) Definition of the function $b$.}
 For all $y,y',z,z' \in \Delta_X(x)$,  by {(i) and} \eqref{equation:cocycle property}, we have that $c_f^{\pm}(y,z)=c_f^{\pm}(y',z')$ whenever $y_{B(r^*)} = y'_{B(r^*)}$ and $z_{B(r^*)} = z'_{B(r^*)}$. Therefore we can extend $c_f^{\pm}$ by continuity to $\overline{\Delta_X(x)}^2 = X^2$. Let us denote by $\overline{c}_f^{+}$ the continuous extension of $c_f^{+}$ to $X^2$.

For any $y\in X$, define $b(y) \coloneqq \overline{c}_f^{+}(y,x)$.
Then we have by definition  of $c_f^{+}$ and by continuity of $b$:
\begin{equation}\label{eq:bbx}
\forall y,y'\in \overline{\Delta_X(x)},\quad \overline{c}_f^+(y,y')=b(y)b(y')^{-1}.
\end{equation}

\paragraph*{(iii) The function $y \mapsto b(\sigma^{\boldsymbol{e}^1}(y))^{-1} \cdot f(y) \cdot b(y)$ is constant on $X$.} 
From the definition it is straightforward that:
\[\forall y,y' \in \Delta_X(x),\quad \overline{c}_f^{+} (y,y') = f(y)^{-1} \cdot \overline{c}_f^{+} (\sigma^{\boldsymbol{e}^1}(y),\sigma^{\boldsymbol{e}^1}(y')) \cdot f(y').\]
By continuity this is also true for all $y,y'\in \overline{\Delta_X(x)}$.
 By \eqref{eq:bbx} this leads to:
\[\forall y,y' \in \overline{\Delta_X(x)},\quad b(y) \cdot b(y')^{-1} = f(y)^{-1} \cdot b(\sigma^{\boldsymbol{e}^1}(y)) \cdot b(\sigma^{\boldsymbol{e}^1}(y'))^{-1} \cdot f(y')\]
which immediately implies that $ 
{y} \mapsto b(\sigma^{\boldsymbol{e}^1}(y))^{-1} \cdot f(y) \cdot b(y)$ is constant on $X$.
The proof is complete.
\end{proof}

The following result is similar to Theorem \textbf{3.2} in~\cite{Schmidt95}:

\begin{theorem}\label{theorem.schmidt.mult}
Let us consider a two-dimensional subshift $X$ and a discrete group $\mathbb{G}$. 
{Assume that there is some $x \in X$ such that $X = \overline{\Delta_X(x)}$ and that $X$ has the strip-gluing property relative to $x$.}
Assume additionally that $\sigma^{\boldsymbol{e}^1}$ is mixing on $X$. Then every continuous cocycle $c : \mathbb{Z}^2 \times X \rightarrow \mathbb{G}$ {is trivial}.
\end{theorem}

\begin{proof}
Let $c : \mathbb{Z}^2 \times X \rightarrow \mathbb{G}$ be a {continuous} cocycle and set $f = c(\boldsymbol{e}^1,\cdot)$. Since $f$ is continuous, there exists $r^{*} \ge 0$ such that if $y,y'$ coincide on {$\boldsymbol B(r^*)$}, then $f(y) = f(y')$. 

\paragraph*{(i) Equality of ${c_f^+}$ and ${c_f^{-}}$.}
 Using the cocycle equation \eqref{eq;cocycle-def} first with $\boldsymbol{m} = \boldsymbol{e}^2$ and $\boldsymbol{n}=\boldsymbol{e}^1$, and then with $\boldsymbol{m} = \boldsymbol{e}^1$ and $\boldsymbol{n}=\boldsymbol{e}^2$, we get that for all $z\in X$:
\[c(\boldsymbol{e}^2 + \boldsymbol{e}^1, z) = c(\boldsymbol{e}^2,\sigma^{\boldsymbol{e}^1}(z)) \cdot c(\boldsymbol{e}^1,z) = c(\boldsymbol{e}^2,\sigma^{\boldsymbol{e}^1}(z)) \cdot f(z).\]
\[c(\boldsymbol{e}^2 + \boldsymbol{e}^1, z) = c(\boldsymbol{e}^1,\sigma^{\boldsymbol{e}^2}(z)) \cdot c(\boldsymbol{e}^2,z) = f(\sigma^{\boldsymbol{e}^2}(z)) \cdot c(\boldsymbol{e}^2,z).\]
We thus get:
\[\forall z \in X,\quad c(\boldsymbol{e}^2,\sigma^{\boldsymbol{e}^1}(z)) \cdot f(z) \cdot c(\boldsymbol{e}^2,z)^{-1} = f(\sigma^{\boldsymbol{e}^2}(z)).\]
From this we have that for all $l \in \mathbb{Z}$ and all {$z \in X$}, 
\begin{equation}\label{eq:conjcocycle}c(\boldsymbol{e}^2,\sigma^{(l+1)\boldsymbol{e}^1}(z)) \cdot f(\sigma^{l\boldsymbol{e}^1}(z)) \cdot c(\boldsymbol{e}^2,\sigma^{l\boldsymbol{e}^1}(z))^{-1} = f(\sigma^{\boldsymbol{e}^2 + l\boldsymbol{e}^1}(z)).
\end{equation}
We claim the following:
\begin{equation}\label{eq:compcocycl}\forall y,y' \in {\Delta_X(x)},\quad c_f^{\pm} (y,y') = c(\boldsymbol{e}^2,y)^{{-1}} \cdot c_f^{\pm} (\sigma^{\boldsymbol{e}^2}(y),\sigma^{\boldsymbol{e}^2}(y')) \cdot c(\boldsymbol{e}^2,y').
\end{equation}

 Let us prove this for $c_f^+$ (the reasoning is similar for $c_f^{-}$). Fix $(y,y') \in \Delta_X$. Let us recall that : 
\[c_f^{+}(y,y') \coloneqq \left(\prod_{k=0}^{m} f\left(\sigma^{k\boldsymbol{e}^1}(y)\right)^{{-1}}\right) \cdot \left(\prod_{k=0}^{m} f\left(\sigma^{k \boldsymbol{e}^1}(y')\right)^{{-1}}\right)^{-1}\]
for all $m\geq m(y,y')$. Let us denote the first factor by $F_1$ and the second factor by $F_2$. In a similar manner, we denote by $F'_1$ and $F'_2$ the two factors in the definition of $c_f^{+} (\sigma^{\boldsymbol{e}^2}(y),\sigma^{\boldsymbol{e}^2}(y'))$.
Taking the inverse of
\eqref{eq:conjcocycle}:
\[c(\boldsymbol{e}^2,\sigma^{l\boldsymbol{e}^1}(z)) \cdot f(\sigma^{l\boldsymbol{e}^1}(z))^{-1} \cdot c(\boldsymbol{e}^2,\sigma^{(l+1)\boldsymbol{e}^1}(z))^{-1} = f(\sigma^{\boldsymbol{e}^2 + l\boldsymbol{e}^1}(z))^{-1}.
\]
We take the product of these equations for $z=y$ over 
{$l=0,1,\ldots,m$}
and obtain: 
\[c(\boldsymbol{e}^2,y) \cdot F_1 \cdot c(\boldsymbol{e}^2,\sigma^{(m+1)\boldsymbol{e}^1}(y))^{-1} = F'_1.\]
In a similar way we have: 
\[c(\boldsymbol{e}^2,\sigma^{(m+1)\boldsymbol{e}^1}(y')) \cdot F_2 \cdot c(\boldsymbol{e}^2,y')^{-1} = F'_2.\]
For $m$ large enough, we have {by continuity of $c$}: 
\[c(\boldsymbol{e}^2,\sigma^{(m+1)\boldsymbol{e}^1}(y)) = c(\boldsymbol{e}^2,\sigma^{(m+1)\boldsymbol{e}^1}(y')).\]
By the three equations above, we get 
\[F'_1 \cdot F'_2 = c(\boldsymbol{e}^2,y) \cdot F_1 \cdot F_2 \cdot c(\boldsymbol{e}^2,y')^{-1},\]
which proves the claim 
\eqref{eq:compcocycl}.

If we fix any $l \ge 1$ and apply
\eqref{eq:compcocycl} inductively, we get:

\[c_f^{\pm} (y,y') = \left(\prod_{q=l-1}^{0} c(\boldsymbol{e}^2,\sigma^{q\boldsymbol{e}^2}(y)) \right)^{ -1} \cdot c_f^{\pm} (\sigma^{l\boldsymbol{e}^2}(y),\sigma^{l\boldsymbol{e}^2}(y')) \cdot \left(\prod_{q=l-1}^{0} c(\boldsymbol{e}^2,\sigma^{q\boldsymbol{e}^2}(y')) \right)\]

Since $(y,y') \in \Delta_X$, for $l\geq 1$ sufficiently large  the points $\sigma^{l\boldsymbol{e}^2}(y)$ and $\sigma^{l\boldsymbol{e}^2}(y')$ coincide on $\mathbb{Z} \times \llbracket -r^* ,r^*\rrbracket $ and thus $c_f^{\pm} (\sigma^{l\boldsymbol{e}^2}(y),\sigma^{l\boldsymbol{e}^2}(y')) = 1_{\mathbb{G}}$.
This implies that $c_f^{+}(y,y') = c_f^{-}(y,y')$.  

\paragraph*{(ii) Definition of the cocycle $\boldsymbol{c'}$.} Since $c_f^{+} = c_f^{-}$, we apply Proposition~\ref{prop.main} to obtain a continuous function $b \colon X \rightarrow \mathbb{G}$ and $a \in \mathbb{G}$ such that 
\begin{equation}\label{eq:ak}
b(\sigma^{\boldsymbol{e}^1}(y))^{-1} \cdot c(\boldsymbol{e}^1,y) \cdot b(y)=a
\end{equation}
for every $y\in X$. We define a function $c'\colon \mathbb{Z}^2 \times X\to \mathbb{G}$ by setting 
\[c'(\boldsymbol{m},y) = b(\sigma^{\boldsymbol{m}}(y))^{-1} \cdot c(\boldsymbol{m},y) \cdot
b(y)\] 
for all $\boldsymbol{m} \in \mathbb{Z}^2$. It is easy to verify that $c'$ is a cocycle and by definition it is cohomologous to $c$. {By \eqref{eq:ak} we have $c'(\boldsymbol{e}^1,y)=a$ for {every} $y\in X$.}
Recall that by the cocycle equation \eqref{eq;cocycle-def} we have

\[
c'(\boldsymbol{m},\sigma^{\boldsymbol{e}^1}(y))c'(\boldsymbol{e}^1,y)=c'(\boldsymbol{e}^1,\sigma^{\boldsymbol{m}}(y))c'(\boldsymbol{m},y)
\]
and therefore, by applying \eqref{eq:ak}
{on $y$ and $\sigma^{\boldsymbol{m}}(y)$}, we obtain that
{\[c'(\boldsymbol{m},\sigma^{\boldsymbol{e}^1}(y)) = a \cdot c'(\boldsymbol{m},y)\cdot a^{-1} \]}
for every $y\in X$.
Applying this formula recursively, for every $l\geq 1$ we have:
{\begin{equation} c'(\boldsymbol{m},\sigma^{l\boldsymbol{e}^1}(y))  = a ^{l} \cdot c'(\boldsymbol{m},y)\cdot a^{-l}
\label{eq:cmxakl}
\end{equation}}

for all $y \in X$ and $\boldsymbol{m} \in \mathbb{Z}^2$.
By symmetry, and because this equation is trivially true for $l=0$, it extends onto all $l\in \mathbb{Z}$.

\paragraph*{(iii) Consequences of the mixing assumption.}
Let us fix some
$\boldsymbol{m} \in \mathbb{Z}^2$ and set $g = c'(\boldsymbol{m},\cdot)$. Let $x_1, x_2\in X$. 

Since $g$ is continuous, for each $i=1,2$ there is an open set $\mathcal{O}_i$ that contains $x_i$ such that for all $x \in \mathcal{O}_i$, $g(x_i) = g(x)$.
Since $\sigma^{\boldsymbol{e}^1}$ is topologically mixing, there exists $l>0$ such that $\sigma^{l\boldsymbol{e}^1}(\mathcal{O}_1)\cap \mathcal{O}_i\neq \emptyset$ for $i=1,2$. Therefore there exist $y_1, y_2 \in \mathcal{O}_1$ such that 
$g(x_i) = g(\sigma^{l\boldsymbol{e}^1}(y_i))$.

 Since both $y_1,y_2$ are in $\mathcal{O}_1$, $g(y_1) = g(y_2)$. Furthermore, by \eqref{eq:cmxakl} we have $g(\sigma^{l\boldsymbol{e}^1}(y_i)) = a ^{l} \cdot g(y_i)\cdot a^{-l}$.  This implies that $g(x_1) = g(x_2)$. Since this holds for all $x_1, x_2\in X$, $g$ is constant on $X$.
 {We have just proved that $c$ is continuously cohomologous to the cocycle $c'$, which is such that $c'(\mathbf{m},\cdot)$ is constant for every $\mathbf{m}$. It is equivalent to say that $c$ is trivial. The proof is thus complete.}
\end{proof}

We are now able to prove the following:
\begin{proposition}\label{theorem.finite.case}
    Let $G$ be a finite undirected graph. If the even square group of $G$ is trivial and $X^2_G$ is mixing then every continuous cocycle on $X^2_G$ with values in a discrete group is trivial.
\end{proposition}

\begin{proof}
The proof is done by reduction to Theorem \ref{theorem.schmidt.mult}. Since $X^2_G$ is mixing, $(X^2_G,\sigma^{\boldsymbol{e}^1})$ is mixing. By Lemma \ref{lemma.mixing.bipartite} and Proposition \ref{proposition.equivalence.cover} and Lemma~\ref{lemma.specification}, there exists $x \in X^2_G$ such that $X$ has the strip-gluing property relative to $x$ and such that $\Delta_{X^2_G}(x)$ is dense.  Thus $X^2_G$ satisfies the assumptions in Theorem \ref{theorem.schmidt.mult} and this completes the proof.
\end{proof}

\section[Proof of Theorem 1.1 \texorpdfstring{$ (\Leftarrow)$}{(<=)}, higher dimensions]{Proof of Theorem \ref{theorem: cohomological triviality} ($\Leftarrow$), higher dimensions\label{section.cohomology.higher.dim}}

In Section \ref{section.cocycle.infinite}, we have provided a proof of direction $(\Leftarrow)$ in Theorem \ref{theorem: cohomological triviality} for two-dimensional homshifts. In this section we extend this result to higher dimensions. This generalization is a direct consequence of Proposition \ref{proposition.non.trivial.sqgrp.cocycle} and the {following theorem, whose proof 
can be found at the end of the section.}

\begin{theorem}\label{theorem.higherd.homshift.trivial}
 If $G$ is such that $X^2_G$ has trivial cohomology and is mixing, then for $d > 2$, $X^d_G$ has also trivial cohomology.
\end{theorem}

This proof relies on two facts: if all two-dimensional projective strips of a higher-dimensional homshift (introduced in Section \ref{section.proj.subdynamics} \cite{MR3314811}) are cohomologically trivial, then the homshift itself is also cohomologically trivial (Proposition \ref{lemma.induction.d-1.to.d}); and if the two-dimensional homshift is cohomologically trivial, it is also the case of any two-dimensional projective strip of any higher-dimensional homshift on the same graph (Proposition \ref{thm:eventoeven}). Though not explicitly mentioned in this exact form, we believe this was the main idea behind \cite[Remark \textbf{10.4}]{Schmidt95}. 

\subsection{Projective subdynamics of homshifts\label{section.proj.subdynamics}}

\begin{definition}
    For a $d$-dimensional subshift $X$ on alphabet $\mathcal{A}$ with $d > 1$ and $n \ge 1$, we denote by \notationidx{$\rho_n(X)$}{$n$-width projective subdynamics of a subshift $X$}$\ \subset \left(\mathcal{A}^{n}\right)^{\mathbb{Z}^{d-1}}$ the $(d-1)$-dimensional subshift that contains all configurations $((x_{\boldsymbol{u}_1,\ldots, \boldsymbol{u}_{d-1},k})_{0\leq k < n})_{ \boldsymbol{u} \in \mathbb{Z}^{d-1}}$ for $x\in X$. We call this subshift the projective $n$-strip of $X$.
\end{definition}

\begin{notation}
    For all $n \ge 1$, denote by \notationidx{$G^n$}{graph of the walks of length $n-1$ on $G$} the graph whose vertices are walks of length $n-1$ on $G$, and which has an edge between $p$ and $q$ when there is an edge in $G$ between $p_i$ and $q_i$ for all $i \in \llbracket 0, n-1 \rrbracket$. This definition is illustrated on Figure~\ref{figure.walk.graph}.
\end{notation}

The following straightforward result implies that for all $n$, the projective $n$-strip of a homshift is a homshift. Details are left to the reader.

\begin{lemma}
    For every finite undirected graph $G$ and all $d > 1$, $\rho_n(X^d_G)$ and $X^{d-1}_{G^{n}}$ are conjugate.
\end{lemma}

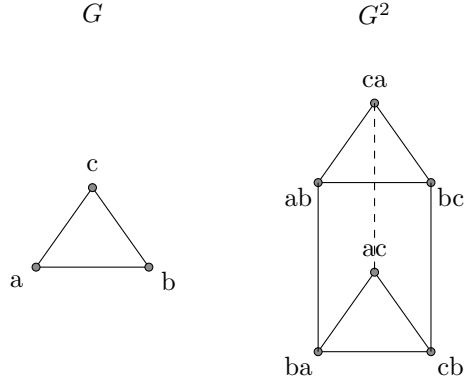
\begin{figure}[!ht]
\begin{center}
\begin{tikzpicture}[scale=0.75]
    \node at  (-0.35,-0.25) {a};
    \node at  (2.35,-0.25) {b};
    \node at  (1,1.8) {c};
        \draw (0,0) -- (2,0) -- (1,1.41) -- (0,0);
        \draw[fill=gray!90] (0,0) circle (2pt);
        \draw[fill=gray!90] (2,0) circle (2pt);
        \draw[fill=gray!90] (1,1.41) circle (2pt);
        \node at (1,4.5) {$G$};
        \begin{scope}[xshift=5cm,yshift=1.5cm]
            \draw (0,0) -- (2,0) -- (1,1.41) -- (0,0);
            \node at  (-0.35,-0.25) {ab};
    \node at  (2.35,-0.25) {bc};
    \node at  (1,1.8) {ca};
        \draw[fill=gray!90] (0,0) circle (2pt);
        \draw[fill=gray!90] (2,0) circle (2pt);
        \draw[fill=gray!90] (1,1.41) circle (2pt);
        \end{scope}
        \begin{scope}[xshift=5cm,yshift=-1.5cm]
            \draw (0,0) -- (2,0) -- (1,1.41) -- (0,0);
            \node at  (-0.35,-0.25) {ba};
    \node at  (2.35,-0.25) {cb};
    \node at  (1,1.8) {ac};
        \draw[fill=gray!90] (0,0) circle (2pt);
        \draw[fill=gray!90] (2,0) circle (2pt);
        \draw[fill=gray!90] (1,1.41) circle (2pt);
        \end{scope}
        \begin{scope}[xshift=5cm,yshift=1.5cm]
            \draw[dashed] (1,1.41) -- (1,1.41-3);
            \draw (0,0) -- (0,-3);
            \draw (2,0) -- (2,-3);
            \node at (1,3) {$G^2$};
        \end{scope}
    \end{tikzpicture}
\end{center}
\caption{Illustration of the definition of $G^n$ for $n=2$.}
\label{figure.walk.graph}
\end{figure}

\subsection{Proof of Theorem \ref{theorem.higherd.homshift.trivial}
\label{section.from.2tod}}

\begin{notation}
    For any $n,m \ge 0$ and any walk $w$ of length $m$ on $G^{n+1}$, the \textbf{dual} of $w$, denoted by \notationidx{$w^{*}$}{dual walk of a walk $w$}, is the walk of length $n$ on $G^{m+1}$ such that for all $i \in \llbracket 0 , {m}\rrbracket$ and $j \in \llbracket 0, n\rrbracket$, $(w_i)_j = (w^{*}_{j})_{i}$.
    This definition is illustrated on Figure~\ref{figure.cutting.gamma}.
\end{notation}

\begin{proposition}\label{thm:eventoeven}
    If the even square group of $G$ is trivial, then for all $n > 1$, the even square group of $G^n$ is also trivial.
\end{proposition}

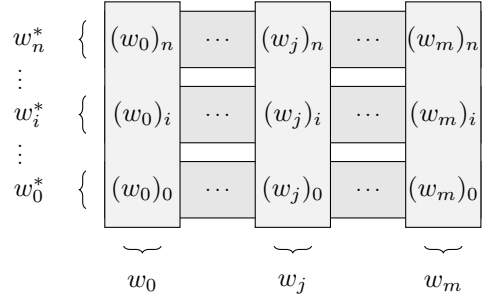
\begin{figure}[!ht]
    \centering
        \begin{tikzpicture}[scale=0.5]
        \node at (0,0) {$w^{*}_0$};
        \node[scale=0.8] at (-0.25,1.1) {$\vdots$};
        \node at (0,2) {$w^{*}_i$};
        \node[scale=0.8] at (-0.25,3.1) {$\vdots$};
        \node at (0,4) {$w^{*}_{n}$};

        \draw[decorate, decoration={brace, amplitude=2.5pt}] (1.5,3.5) -- (1.5,4.5);
        \draw[decorate, decoration={brace, amplitude=2.5pt}] (1.5,1.5) -- (1.5,2.5);
        \draw[decorate, decoration={brace, amplitude=2.5pt}] (1.5,-0.5) -- (1.5,0.5);

        \draw[fill=gray!20] (2,-0.75) rectangle (12,0.75);
        \draw[fill=gray!20] (2,1.25) rectangle (12,2.75);
        \draw[fill=gray!20] (2,3.25) rectangle (12,4.75);

         \draw[fill=gray!10] (2,-1) rectangle (4,5);
        \draw[fill=gray!10] (6,-1) rectangle (8,5);
        \draw[fill=gray!10] (10,-1) rectangle (12,5);
        
        \node at (3,4) {$(w_0)_n$};
        \node at (3,2) {$(w_0)_i$};
        \node at (3,0) {$(w_0)_0$};

        \node[scale=0.8] at (5,4) {$\hdots$};
        \node[scale=0.8] at (5,2) {$\hdots$};
        \node[scale=0.8] at (5,0) {$\hdots$};

         \node[scale=0.8] at (9,4) {$\hdots$};
        \node[scale=0.8] at (9,2) {$\hdots$};
        \node[scale=0.8] at (9,0) {$\hdots$};

        \node at (7,4) {$(w_j)_n$};
        \node at (7,2) {$(w_j)_i$};
        \node at (7,0) {$(w_j)_0$};

         \node at (11,4) {$(w_{m})_n$};
        \node at (11,2) {$(w_{m})_i$};
        \node at (11,0) {$(w_{m})_0$};

        \draw[decorate, decoration={brace, amplitude=2.5pt,mirror}] (2.5,-1.5) -- (3.5,-1.5);
        \draw[decorate, decoration={brace, amplitude=2.5pt,mirror}] (6.5,-1.5) -- (7.5,-1.5);
        \draw[decorate, decoration={brace, amplitude=2.5pt,mirror}] (10.5,-1.5) -- (11.5,-1.5);

        \node at (3,-2.5) {$w_0$};
        \node at (7,-2.5) {$w_j$};
        \node at (11,-2.5) {$w_{m}$};
    \end{tikzpicture}
    \caption{Illustration of the definition of dual walk.}
    \label{figure.cutting.gamma}
\end{figure}

\begin{proof}
Consider $\gamma$ a cycle of even length in $G^n$, and fix $s$ a cycle of length 2 on $\gamma_0$ in $G^n$. Apply Lemma~\ref{lemma.box} (see also Remark~\ref{remark:any_t}) to the cycle $\gamma^{*}_{n-1}$ and graph $G$, and to avoid conflict of notation denote by $\tilde {k}$, $\tilde{n}$ and $R'$ the integers $k$ and $n$ and the pattern $R$ given by the lemma. Set $n_+ \coloneqq n + 2 \tilde{n}$. Let $c'_0 , \ldots , c'_{n-1}$ be the dual walk of the cycle $s^{\tilde{k}} \odot \gamma \odot s^{\tilde{k}}$. Recall that for any cycle $p$, $\omega(p)= p_1 \ldots p_{l(p)} p_1$ is the circular shift of $p$. Note that for each $i$, 
$c'_{i} = t_i ^{\tilde{k}} \odot \gamma^{*}_{i} \odot t_i^{\tilde{k}}$,  
where $t_i$ is some cycle of length 2.

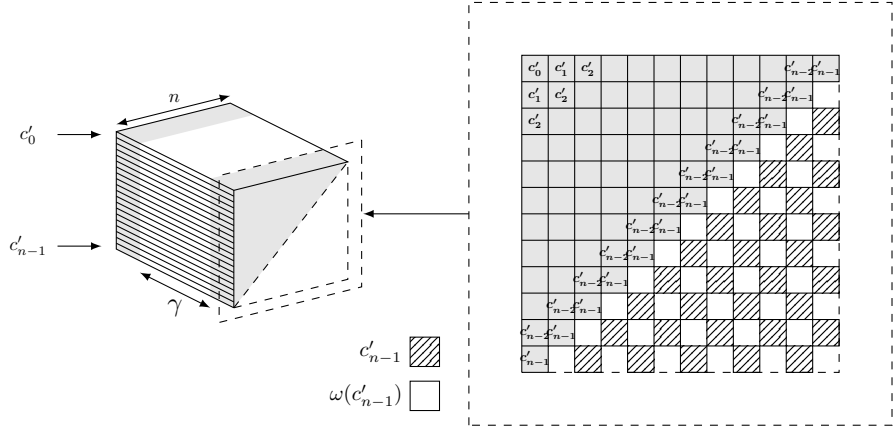
\begin{figure}[!ht]
    \centering
    \begin{tikzpicture}[scale=0.625]
    \begin{scope}[scale=1.25]
    \fill[gray!20] (-2,3) -- (-2,1) -- (-1.6,0.8) -- (-1.6,2.8) -- (2.4 - 4 + 1.94,3.8 - 1 + 0.485) -- (2-4 + 1.94,4-1 + 0.485) -- (-2,3);
    \fill[gray!20] (0,0) -- (0,2) -- (1.94,2+0.485) -- (0,0);
    \begin{scope}[xshift=1.6cm,yshift=-0.8cm]
        \fill[gray!20] (-2,3) -- (-2,1) -- (-1.6,0.8) -- (-1.6,2.8) -- (2.4 - 4 + 1.94,3.8 - 1 + 0.485) -- (2-4 + 1.94,4-1 + 0.485) -- (-2,3);
    \end{scope}
    \draw (0,0) -- (0,2) -- (-2,3) -- (-2,1) -- (0,0);
    \foreach \x in {1,...,19} {
    \draw (0,0.1*\x) -- (-2,1+0.1*\x);
    }
    \draw[-latex] (-3,1.05) -- (-2.25,1.05);
    \node[scale=0.7] at (-3.5,1.05) {$c'_{n-1}$};
    \node[scale=0.7] at (-3.5,2.95) {$c'_0$};
    \draw[-latex] (-3,2.95) -- (-2.25,2.95);
    \draw[dashed] (0,0) -- (1.94,2+0.485);
    \draw[dashed] (0,0) -- (1.94,0.485) -- (1.94,2+0.485);

    \draw[latex-latex] (-1.6,0.6) -- (-0.4,0);
    \node at (-1,0) {$\gamma$};

    \draw[dashed] (1.25*1.94-0.25,2+1.25*0.485+0.25) -- (1.25*1.94-0.25,2+1.25*0.485+0.25-2.5) -- (-0.25,-0.25) -- (-0.25,2.25) -- (1.25*1.94-0.25,2+1.25*0.485+0.25);

    \draw[latex-latex] (-2,3+0.125) -- (-2+1.94,3+0.485+0.125);
    \node[scale=0.7] at (-2+0.99,3+0.2425+0.325) {$n$};
    
    \draw (-2,3) -- (-2+1.94,3+0.485) -- (1.94,2+0.485) -- (0,2);

    \draw[pattern=north east lines] (3,-1) rectangle (3.5,-0.5);
    \node[scale=0.8] at (2.5,-0.75) {$c'_{n-1}$};
    \draw (3,-1.25) rectangle (3.5,-1.75);
    \node[scale=0.8] at (2.25,-1.5) {$\omega(c'_{n-1})$};

    \end{scope}
    \begin{scope}[xshift=5cm,yshift=-2.5cm,scale=2.25]
    \draw[-latex] (0,2) -- (-1,2);
        \draw[dashed] (0,0) rectangle (4,4);
        \draw[fill=gray!20] (0.5,0.5) -- (0.75,0.5) -- (0.75,0.75) -- (1,0.75) -- (1,1) -- (1.25,1) -- (1.25,1.25) -- (1.5,1.25) -- (1.5,1.5) -- (1.75,1.5) -- (1.75,1.75) -- (2,1.75) -- (2,2) -- (2.25,2) -- (2.25,2.25) -- (2.5,2.25) -- (2.5,2.5) -- (2.75,2.5) -- (2.75,2.75) -- (3,2.75) -- (3,3) -- (3.25,3) -- (3.25,3.25) -- (3.5,3.25) -- (3.5,3.5) -- (0.5,3.5) -- (0.5,0.5);
        \foreach \x in {0.75,1,1.25,1.5,1.75,2,2.25,2.5,2.75,3,3.25} {
        \draw (\x,3.5) -- (\x,\x);
        \draw (0.5,\x) -- (\x,\x);
        }
        \draw[dashed] (0.75,0.5) -- (3.5,0.5) -- (3.5,3.25);
        \foreach \x in {0,1,2,3,4,5,6,7,8,9,10} {
        \node[scale=0.5] at (0.625+0.25*\x,0.625+0.25*\x) {$\boldsymbol{c'_{n-1}}$};
        \node[scale=0.5] at (0.625+0.25*\x,0.875+0.25*\x) {$\boldsymbol{c'_{n-2}}$};
        }
        \node[scale=0.5] at (0.625+0.25*11,0.625+0.25*11) {$\boldsymbol{c'_{n-1}}$};
        \node[scale=0.5] at (0.625,0.625+0.25*11) {$\boldsymbol{c'_{0}}$};
        \node[scale=0.5] at (0.625,0.625+0.25*10) {$\boldsymbol{c'_{1}}$};
        \node[scale=0.5] at (0.625+0.25,0.625+0.25*11) {$\boldsymbol{c'_{1}}$};
        \node[scale=0.5] at (0.625,0.625+0.25*9) {$\boldsymbol{c'_{2}}$};
        \node[scale=0.5] at (0.625+0.25,0.625+0.25*10) {$\boldsymbol{c'_{2}}$};
        \node[scale=0.5] at (0.625+2*0.25,0.625+0.25*11) {$\boldsymbol{c'_{2}}$};
        \foreach \x in {0,1,2,3,4} {
        \draw[pattern=north east lines] (1+0.5*\x,0.5) rectangle (1.25+0.5*\x,0.75);
        }
        \foreach \x in {0,1,2,3,4} {
        \draw[pattern=north east lines] (1+0.5*\x+0.25,0.5+0.25) rectangle (1.25+0.5*\x+0.25,0.75+0.25);
        }
        \foreach \x in {0,1,2,3} {
        \draw[pattern=north east lines] (1+0.5*\x+0.5,0.5+0.5) rectangle (1.25+0.5*\x+0.5,0.75+0.5);
        }
        \foreach \x in {0,1,2,3} {
        \draw[pattern=north east lines] (1+0.5*\x+0.75,0.5+0.75) rectangle (1.25+0.5*\x+0.75,0.75+0.75);
        }
        \foreach \x in {0,1,2} {
        \draw[pattern=north east lines] (1+0.5*\x+1,0.5+1) rectangle (1.25+0.5*\x+1,0.75+1);
        }
        \foreach \x in {0,1,2} {
        \draw[pattern=north east lines] (1+0.5*\x+1.25,0.5+1.25) rectangle (1.25+0.5*\x+1.25,0.75+1.25);
        }
        \foreach \x in {0,1} {
        \draw[pattern=north east lines] (1+0.5*\x+1.5,0.5+1.5) rectangle (1.25+0.5*\x+1.5,0.75+1.5);
        }
        \foreach \x in {0,1} {
        \draw[pattern=north east lines] (1+0.5*\x+1.75,0.5+1.75) rectangle (1.25+0.5*\x+1.75,0.75+1.75);
        }
        \draw[pattern=north east lines] (1+2,0.5+2) rectangle (1.25+2,0.75+2);
        \draw[pattern=north east lines] (1+2.25,0.5+2.25) rectangle (1.25+2.25,0.75+2.25);
    \end{scope}
    \end{tikzpicture}
    \caption{Illustration for Step 1 of the construction of $R$ in the proof of Proposition \ref{thm:eventoeven}. We represent $R$ as a three-dimensional pattern on alphabet $G$, although it is formally a two-dimensional pattern on alphabet $G^n$.
    }
    \label{fig:thm.eventoeven}
\end{figure}

\begin{figure}[!ht]
    \centering
    \begin{tikzpicture}[scale=0.8]
    \begin{scope}[scale=1.25]
    \fill[gray!20] (-2,3) -- (-2,1) -- (-1.6,0.8) -- (-1.6,2.8) -- (2.4 - 4 + 2* 1.94,3.8 - 1 + 2* 0.485) -- (2-4 + 2*1.94,4-1 + 2*0.485) -- (-2,3);
    \fill[gray!20] (0,0) -- (0,2) -- (2*1.94,2+2*0.485) -- (2*1.94,2*0.485) -- (0,0);
    \begin{scope}[xshift=1.6cm,yshift=-0.8cm]
        \fill[gray!20] (-2,3) -- (-2,1) -- (-1.6,0.8) -- (-1.6,2.8) -- (2.4 - 4 + 2*1.94,3.8 - 1 + 2*0.485) -- (2-4 + 2*1.94,4-1 + 2*0.485) -- (-2,3);
    \end{scope}

        \fill[gray!50] (1.94,0.485) -- (-2+1.94,1+0.485) -- (-2+1.94,1+0.485+0.1) -- (-2+2*1.94,1+2*0.485+0.1) -- (2*1.94,2*0.485+0.1) -- (2*1.94,2*0.485) -- (1.94,0.485);
    \draw (0,0) -- (0,2) -- (-2,3) -- (-2,1) -- (0,0);
    \draw (0,0) -- (1.94,0.485) -- (1.94,2+0.485);

    \draw[latex-latex] (-1.6,0.6) -- (-0.4,0);
    \node at (-1,0) {$\gamma$};

    \draw[latex-] (-3+7,1.225) -- (-2.25+7,1.525);
    \draw[latex-] (-3+7,1.025) -- (-2.25+7,0.725);
    \draw[latex-] (-3+7,1.125) -- (-2.25+7,1.125);
    \node[scale=0.6] at (-1.75+7,0.725) {$R'$};
    \node[scale=0.6] at (-1.75+7,1.125) {$\omega(R')$};
    \node[scale=0.6] at (-1.75+7,1.525) {$R'$};
    \node[scale=0.6] at (-2.5+7,1.925) {$\vdots$};

    \draw (-2+1.94,3+0.485) -- (-2+2*1.94,3+2*0.485) -- (2*1.94,2+2*0.485) -- (1.94,2+0.485);

    \draw[dashed] (-2,1) -- (-2+2*1.94,1+2*0.485) -- (2*1.94,2*0.485);
    \draw[dashed] (-2+2*1.94,1+2*0.485) -- (-2+2*1.94,3+2*0.485);
    \draw[dashed] (-2+1.94,1+0.485) -- (-2+1.94,3+0.485);
    \draw[dashed] (-2+1.94,1+0.485) --  (1.94,0.485);
    
    \draw[dashed] (-2+1.94,1+0.485+0.1) --  (1.94,0.485+0.1);
    \draw[dashed] (-2+1.94,1+0.485+0.2) --  (1.94,0.485+0.2);
    \draw[dashed] (-2+1.94,1+0.485+0.3) --  (1.94,0.485+0.3);

    \draw[dashed] (-2+1.94,1+0.485+0.1) --  (-2+2*1.94,1+2*0.485+0.1);
    \draw[dashed] (-2+1.94,1+0.485+0.2) --  (-2+2*1.94,1+2*0.485+0.2);
    \draw[dashed] (-2+1.94,1+0.485+0.3) --  (-2+2*1.94,1+2*0.485+0.3);

    \draw[dashed] (-2+2*1.94,1+2*0.485+0.1) -- (2*1.94,2*0.485+0.1);
    \draw[dashed] (-2+2*1.94,1+2*0.485+0.2) -- (2*1.94,2*0.485+0.2);
    \draw[dashed] (-2+2*1.94,1+2*0.485+0.3) -- (2*1.94,2*0.485+0.3);

    \draw[latex-latex] (-2+2*1.94,3+2*0.485+0.125) -- (-2+1.94,3+0.485+0.125);

    \draw[dashed] (-2 + 1.5* 1.94,3+1.5*0.485) -- (1.5*1.94,2+1.5*0.485);
    \node[scale=0.7] at (-1 + 1.5* 1.94+0.2,3+1*0.485-0.1) {$w_j$}; 
    \draw (-2+1.94,3+0.485) -- (-2+1.94,3+0.485);
    \node[scale=0.7] at (-2+0.99+1.94,3+0.485+0.2425+0.325) {$2n_{*}$};
    
    \draw (-2,3) -- (-2+1.94,3+0.485) -- (1.94,2+0.485) -- (0,2);
    \draw  (1.94,0.485) --  (2*1.94,2*0.485) --  (2*1.94,2+2*0.485);
    \draw  (1.94,0.485+0.1) --  (2*1.94,2*0.485+0.1);
     \draw  (1.94,0.485+0.2) --  (2*1.94,2*0.485+0.2);
      \draw  (1.94,0.485+0.3) --  (2*1.94,2*0.485+0.3);
       \draw  (1.94,0.485+0.4) --  (2*1.94,2*0.485+0.4);
           \draw  (1.94,0.485+0.5) --  (2*1.94,2*0.485+0.5);
     \draw  (1.94,0.485+0.6) --  (2*1.94,2*0.485+0.6);
      \draw  (1.94,0.485+0.7) --  (2*1.94,2*0.485+0.7);
       \draw  (1.94,0.485+0.8) --  (2*1.94,2*0.485+0.8);
              \draw  (1.94,0.485+0.9) --  (2*1.94,2*0.485+0.9);

        \draw  (1.94,0.485+0.1+0.9) --  (2*1.94,2*0.485+0.1+0.9);
     \draw  (1.94,0.485+0.2+0.9) --  (2*1.94,2*0.485+0.2+0.9);
      \draw  (1.94,0.485+0.3+0.9) --  (2*1.94,2*0.485+0.3+0.9);
       \draw  (1.94,0.485+0.4+0.9) --  (2*1.94,2*0.485+0.4+0.9);
           \draw  (1.94,0.485+0.5+0.9) --  (2*1.94,2*0.485+0.5+0.9);
     \draw  (1.94,0.485+0.6+0.9) --  (2*1.94,2*0.485+0.6+0.9);
      \draw  (1.94,0.485+0.7+0.9) --  (2*1.94,2*0.485+0.7+0.9);
       \draw  (1.94,0.485+0.8+0.9) --  (2*1.94,2*0.485+0.8+0.9);
              \draw  (1.94,0.485+0.9+0.9) --  (2*1.94,2*0.485+0.9+0.9);
              \draw  (1.94,0.485+1.9) --  (2*1.94,2*0.485+1.9);

    \end{scope}
   
    \end{tikzpicture}
    \caption{Illustration for Step 2 of the construction of $R$ in the proof of Proposition \ref{thm:eventoeven}. The notation $\omega(R')$ refers to the pattern obtained from $R'$ by applying $\omega$ row by row.}
    \label{fig:thm.eventoeven.2}
\end{figure}
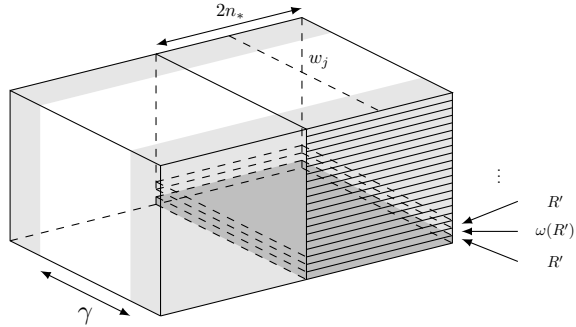

Let us define {in two steps} a rectangular pattern $R$ on alphabet $G^n$ with support $\llbracket 0, \nu\rrbracket \times \llbracket 0 ,n_{+}\rrbracket$, where $\nu\coloneqq 4k_{+}+l(\gamma)$.

\begin{enumerate}[label=Step~\arabic*.]
\item For all $i \in \llbracket  0, n \rrbracket$, the restriction $R|_{\llbracket 0 , \nu \rrbracket \times \{i\}}$ is defined as having the following dual walk: 
\[c'_i , \ldots , c'_{n-1}, d, c, d, c, \ldots\qquad \textrm{where }{c\coloneqq c'_{n-1} \textrm{ and } d\coloneqq\omega(c'_{n-1})}.\]
This step is illustrated in Figure \ref{fig:thm.eventoeven}.

\item 
For all $j \in \llbracket n+1 , n_{+}\rrbracket$, we define the slice $R_{\llbracket 0 , \nu \rrbracket \times \{j\}}$ as having the following dual walk: 
\[w_j, \omega(w_j) , w_j , \omega(w_j), \ldots \qquad 
\text{where }w_j \coloneqq R'_{\llbracket 0 , \nu \rrbracket \times \{j-n\}}.\]
This step is illustrated in Figure \ref{fig:thm.eventoeven.2}.
\end{enumerate}
Note that $R_{\llbracket 0, \nu \rrbracket \times\{0\}} = s^{\tilde{k}}\odot \gamma\odot s^{\tilde{k}}$. 
Furthermore, $R'_{\llbracket 0,\nu\rrbracket\times \{2\tilde{n}\}}$ is two-periodic and $R'_{\{0\}\times \llbracket 0 , 2\tilde{n}\rrbracket} = R'_{\{\nu\}\times \llbracket 0 , 2\tilde{n}\rrbracket}$. 
Thus, $R_{\llbracket 0,\nu\rrbracket\times \{n_{+}\}}$ is two-periodic and $R_{\{0\}\times \llbracket0,n_+\rrbracket}= R_{\{\nu\}\times \llbracket0,n_+\rrbracket}$. By Lemma \ref{lemma.border.sqdec} we have that $\varphi(\partial R)\in \Delta(G^n)[\gamma_0]$. 
Since $\varphi(R_l\odot R_u\odot R_r)$ is trivial, 
$\varphi(s^{\tilde{k}} \odot \gamma \odot s^{\tilde{k}}) = \varphi(\gamma) \in \Delta(G^n)[\gamma_0]$.
This completes the proof.
\end{proof}

\begin{proposition}\label{lemma.induction.d-1.to.d}
    Consider a mixing $d-$dimensional subshift $X$. If $\rho_n(X)$ is cohomologically trivial for all $n \ge 1$, then so is $X$.
\end{proposition}

\begin{proof}
Let $c : \mathbb{Z}^d \times X \rightarrow \mathbb{G}$ be a continuous cocycle. There exists some $k$ such that $(c(\boldsymbol{e}_i,x))_{i \in \llbracket 1, d\rrbracket}$ is determined by the value of $x_{\llbracket - k , k \rrbracket^d}$.

\paragraph*{(i) The projected cocycle $\boldsymbol{\tilde{c}}$ on $\rho_{2n+1}(X)$ is trivial.} 
Fix any $n \geq k$.
We set $\tilde{c} : \mathbb{Z}^{d-1} \times \rho_{2n+1}(X) \rightarrow \mathbb{G}$ as follows: 
\[\tilde{c}(\boldsymbol{v},x_{\mathbb{Z}^{d-1} \times {\llbracket -n, n\rrbracket}}) = c((\boldsymbol{v}_0, \ldots , \boldsymbol{v}_{d-1},0),x)\]
for all $x \in X$ and $\boldsymbol{v} \in \mathbb{Z}^{d-1}$. The map $\tilde{c}$ is well-defined because the values $c(\boldsymbol{e}_i,x)$ are determined by the value  of $x_{\llbracket - k , k \rrbracket^d}$. Note that if $X\subset A^{\Z^d}$ then formally $\rho_{2n+1}(X)$ has alphabet $A^{\llbracket 0, 2n+1 \llbracket}$. However, with an abuse of notation we will center it at $0$ instead and think of it as having alphabet $A^{\llbracket -n, n \rrbracket}$.

Furthermore, it is easy to check that $\tilde{c}$ is a cocycle. Since $\rho_{2n+1}(X)$ is cohomologically trivial, there exists a continuous map $\tilde{b} : \rho_{2n+1}(X) \rightarrow \mathbb{G}$  and a {homomorphism} $\theta : \mathbb{Z}^{d-1} \rightarrow \mathbb{G}$  such that for all $x \in \rho_{2n+1}(X)$ and $\boldsymbol{v} \in \mathbb{Z}^{d-1}$, 
\[\tilde{c}(\boldsymbol{v},x) = \tilde{b}(\sigma^{\boldsymbol{v}}(x))^{-1} \theta(\boldsymbol{v})\tilde{b}(x).\]Finally observe that since $\rho_k(X)$ is cohomologically trivial and the values $c(\boldsymbol{e}_i,x)$ are determined by the value  of $x_{\llbracket - k , k \rrbracket^d}$ we can assume that the values of the transfer function $\tilde b(x)$ is also determined by $x_{\llbracket - l , l \rrbracket^{d-1}\times\llbracket -k,k\rrbracket}$ for some fixed $l\in \N$.

\paragraph*{(ii) A cohomologous cocycle $\boldsymbol{c'}$ on $X$.}
Let us denote by $b:X \rightarrow \mathbb{G}$ such that for all $x \in X$:
\[{b(x) = \tilde{b}(x_{\mathbb{Z}^{d-1} \times {\llbracket -n, n\rrbracket}})}.\]
We denote by $c'$ the cocycle on $X$ given by:
\[c'(\boldsymbol{v},x) = b(\sigma^{\boldsymbol{v}}(x)) c(\boldsymbol{v},x) b(x)^{-1},\]
which is by definition cohomologous to $c$. 

\paragraph*{(iii) The cocycle $c'$ is a homomorphism.}

When $\boldsymbol{v} \in \mathbb{Z}^{d-1} \times \{0\}$, we have $c'(\boldsymbol{v},x) = \theta(\boldsymbol{v})$; this is a slight abuse of notation as $\theta$ is defined on $\Z^{d-1}$, that is, we ignore the last coordinate of $\boldsymbol{v}$.

We are left to prove that for all $t \in \mathbb{Z}$ and $x,x'\in X$, $c'(t\boldsymbol{e}_d,x) = c'(t\boldsymbol{e}_d,x')$. 

Fix arbitrary configurations $x, x'\in X$. Since $X$ is mixing, there exists $r^{*}$ $\geq 2l$ such that for all $r \ge r^{*}$, there exists $z \in X$ such that $z$ coincides with $x$ on $\llbracket - l , l \rrbracket ^{d-1} \times \llbracket -k-|t|, k+|t| \rrbracket$ and with $\sigma^{-r\boldsymbol{e}_1}(x')$ on $r\boldsymbol{e}_1 + (\llbracket - l , l \rrbracket ^{d-1} \times \llbracket -k-|t|, k+|t| \rrbracket)$. The cocycle equation implies: 
\[ c'(t\boldsymbol{e}_d,\sigma^{r\boldsymbol{e}_1}(z)) \cdot c'(r\boldsymbol{e}_1,z) = c'(r\boldsymbol{e}_1,\sigma^{t\boldsymbol{e}_d}(z)) \cdot c'(t\boldsymbol{e}_d,z).\]
 By Remark \ref{remark.nullvalue.cocycle}, we have that $c'(r\boldsymbol{e}_1,\sigma^{t\boldsymbol{e}_d}(z)) = c'(-r\boldsymbol{e}_1,\sigma^{t\boldsymbol{e}_d + r \boldsymbol{e}_1}(z))^{-1}$ and $c'(t\boldsymbol{e}_d,z) = c'(-t\boldsymbol{e}_d,\sigma^{t\boldsymbol{e}_d}(z))^{-1}$, and we rewrite the equation above as:
\[c'(-t\boldsymbol{e}_d,\sigma^{t\boldsymbol{e}_d}(z)) \cdot c'(-r\boldsymbol{e}_1,\sigma^{t\boldsymbol{e}_d + r \boldsymbol{e}_1}(z)) \cdot c'(t\boldsymbol{e}_d,\sigma^{r\boldsymbol{e}_1}(z)) \cdot c'(r\boldsymbol{e}_1,z) = 1_{\mathbb{G}}.\]
By using the remark above, this can be rewritten as:
\[c'(-t\boldsymbol{e}_d,\sigma^{t\boldsymbol{e}_d}(z)) \cdot \theta(-r\boldsymbol{e}_1) \cdot c'(t\boldsymbol{e}_d,\sigma^{r\boldsymbol{e}_1}(z)) \cdot \theta(r\boldsymbol{e}_1) = 1_{\mathbb{G}}.\]
Therefore, by the continuity of $c'$ and our choice of $l$:
\[c'(-t\boldsymbol{e}_d,\sigma^{t\boldsymbol{e}_d}(x)) \cdot \theta(-r\boldsymbol{e}_1) \cdot c'(t\boldsymbol{e}_d,x') \cdot \theta(r\boldsymbol{e}_1) = 1_{\mathbb{G}}.\]
By the cocycle equation we have $c'(-t\boldsymbol{e}_d,\sigma^{t\boldsymbol{e}_d}(x)) = (c'(t\boldsymbol{e}_d,x))^{-1}$ and thus:
\[c'(t\boldsymbol{e}_d,x) = \theta(-r\boldsymbol{e}_1)\cdot c'(t\boldsymbol{e}_d,x')\cdot \theta(r\boldsymbol{e}_1).\]
Since the above equation holds for arbitrary $x$ and $x'$, by applying it on $x = x'$ we obtain for all $r$ large enough:

\[c'(t\boldsymbol{e}_d,x) =  \theta(-r\boldsymbol{e}_1)\cdot c'(t\boldsymbol{e}_d,x')\cdot \theta(r\boldsymbol{e}_1) = c'(t\boldsymbol{e}_d,x').\]
This means precisely that $c'(t\boldsymbol{e}_d,\cdot)$ is constant on $X$.
\end{proof}

\begin{proof}[\textbf{Proof of Theorem} \ref{theorem.higherd.homshift.trivial}]
    It is sufficient to prove that, if $X^2_G$ is cohomologically trivial and mixing, then $X^d_G$ is cohomologically trivial for all $d > 2$. Let us prove the base case $(\mathcal{P}_3)$.
    Assume that $X^2_G$ is cohomologically trivial. By Proposition~\ref{proposition.non.trivial.sqgrp.cocycle}, all even length cycles on $G$ are square-decomposable. By Proposition~\ref{thm:eventoeven}, this is also the case for all even length cycles of $G^n$ for all $n \ge 1$. Since $X^2_G$ is mixing, $X^2_{G^n}$ is also mixing for all $n \ge 1$. By Proposition~\ref{theorem.finite.case}, we thus have that $X^2_{G^n}$ is cohomologically trivial for all $n \ge 1$. By Proposition~\ref{lemma.induction.d-1.to.d}, $X^3_{G}$ is cohomologically trivial. The induction step is as follows. Assume $(\mathcal{P}_d)$ for some $d > 2$. For any graph $G$ such that $X^2_G$ is cohomologically trivial, apply $(\mathcal{P}_d)$ to $G^n$ for all $n \ge 1$ to obtain that $\rho_n(X^{d+1}_G)$ is cohomologically trivial for all $n \ge 1$. By Proposition~\ref{lemma.induction.d-1.to.d}, $X^{d+1}_G$ is cohomologically trivial.
\end{proof}

\section{\label{section.boxext}Box-extension property}

The box-extension property, defined below, is known to imply that every continuous cocycle is trivial~\cite{Schmidt95}. Homshifts provide many examples that show that the converse is not true. We provide one such example in this section.
 
\begin{definition}
A two-dimensional subshift $X$ is said to have the \textbf{box-extension property} when there exists $r \ge 1$ such that 
for all $n \ge 0$, every locally admissible pattern on $\boldsymbol{B}^2(n+r+1) \backslash \boldsymbol{B}^2(n+r)$ which can be extended into a locally admissible pattern on $\boldsymbol{B}^2(n+r+1) \backslash \boldsymbol{B}^2(n)$ can be extended into a configuration of $X$.
\end{definition}

\begin{remark}
When $X$ is a homshift, {since every locally admissible pattern on a square can be extended into a configuration of $X$ (\cite[Proposition 2.1]{MR3743365}),} this is equivalent to say that every locally admissible pattern on $\boldsymbol{B}^2(n+r+1) \backslash \boldsymbol{B}^2(n+r)$ which can be extended into a locally admissible pattern on $\boldsymbol{B}^2(n+r+1) \backslash \boldsymbol{B}^2(n)$ can be extended into a locally admissible pattern on $\boldsymbol{B}^2(n+r+1)$.
\end{remark}

We prove in Section~\ref{section.boxextsqdec} that if $X_G^2$ has the box-extension property then $G$ is square-decomposable. Then we prove in Section~\ref{section.nonboxext} that the homshift associated with the Kenkatabami graph does not have the box-extension property. Since this graph has trivial even square group, this implies that the box-extension property is strictly stronger than cohomological triviality. 

\subsection{\label{section.boxextsqdec}Box-extension implies trivial even square group}
The box-extension property implies trivial cohomology and we know by Theorem~\ref{theorem: cohomological triviality} that a homshift has trivial cohomology if and only if its even square group is trivial. Let us see a more combinatorial argument to show that the box-extension property implies that the even square group is trivial.

\begin{proposition}\label{proposition.completion.border}
    For all {cycles $c,c',c'',c'''$ cycles of length $2n$ which all begin and end at the same vertex $a$ of $G$}, and for every $k \ge 0$, there exists a pattern on $\textbf{B}^2(n(2k+1)) \backslash \textbf{B}^2(n-1)$
    whose restriction $p$ on $\partial \textbf{B}^2(n)$, is {defined by 
    $p_l = c, p_u = c', p_r = c'', p_d = c'''$ (following Notation \ref{notation.clockwise}),}
    and whose restriction $q$ to $\partial \textbf{B}^2(n(2k+1))$ is {defined by 
    $q_l = c^{-k} \odot c^{k+1}$, 
    $q_u = {(c')}^{k+1} \odot {(c')}^{-k}$, 
    $q_r = {(c'')}^{-k} \odot {(c'')}^{k+1}$ and $q_d = {(c''')}^{k+1} \odot {(c''')}^{-k}$,}
    {as on the left hand side of Figure \ref{figure.completion.border}.}
\end{proposition}

\begin{proof}
    Provided an integer $m$ and two cycles $\gamma,\gamma'$ of length $m$ which begin and end on the same vertex, let $P_\lrcorner(\gamma,\gamma')$ be the pattern on $\llbracket 0 , m \rrbracket ^2$ such that :

    \begin{itemize}
        \item { $\left(P_\lrcorner(\gamma,\gamma')\right)_l = \gamma$ and $\left(P_\lrcorner(\gamma,\gamma')\right)_d = \gamma^{-1}$;} 
        \item { $\left(P_\lrcorner(\gamma,\gamma')\right)_u = \gamma'$ and $\left(P_\lrcorner(\gamma,\gamma')\right)_r = (\gamma')^{-1}$;} 
        \item for all $\textbf{j} \in \llbracket 1, m\rrbracket \times \llbracket 1, m \rrbracket$, $P_\lrcorner(\gamma,\gamma')_{\textbf{j}} = P_\lrcorner(\gamma,\gamma')_{\textbf{j} - (1,1)}$.
    \end{itemize} 

The pattern $P_\urcorner(\gamma,\gamma')$ is obtained from $P_\lrcorner(\gamma,\gamma')$ by rotation of angle $-\pi/2$. {These two notations are illustrated on Figure \ref{figure.p_corner}.}

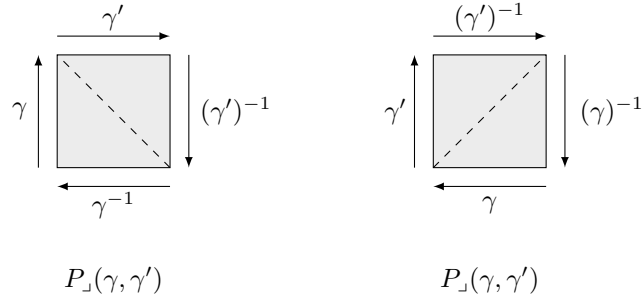
\begin{figure}[!ht]
    \centering
    \begin{tikzpicture}[scale=0.5]
        \draw[fill=gray!15] (0,0) rectangle (3,3);
        \draw[-latex] (-0.5,0) -- (-0.5,3);
        \draw[latex-] (3.5,0) -- (3.5,3);
        \draw[-latex] (0,3.5) -- (3,3.5);
        \draw[latex-] (0,-0.5) -- (3,-0.5);
        \draw[dashed] (3,0) -- (0,3);
        \node at (1.5,-1) {$\gamma^{-1}$};
        \node at (1.5,4) {$\gamma'$};
        \node at (-1,1.5) {$\gamma$};
        \node at (4.75,1.5) {$(\gamma')^{-1}$};
        \node at (1.5,-3) {$P_\lrcorner(\gamma,\gamma')$};
        \begin{scope}[xshift=10cm]
             \draw[fill=gray!15] (0,0) rectangle (3,3);
             \draw[-latex] (-0.5,0) -- (-0.5,3);
        \draw[latex-] (3.5,0) -- (3.5,3);
        \draw[-latex] (0,3.5) -- (3,3.5);
        \draw[latex-] (0,-0.5) -- (3,-0.5);
        \draw[dashed] (0,0) -- (3,3);
        \node at (1.5,-1) {$\gamma$};
        \node at (1.5,4) {$(\gamma')^{-1}$};
        \node at (-1,1.5) {$\gamma'$};
        \node at (4.75,1.5) {$(\gamma)^{-1}$};
        \node at (1.5,-3) {$P_\lrcorner(\gamma,\gamma')$};
        \end{scope}
    \end{tikzpicture}
    \caption{Illustration for the notations $P_\lrcorner(\gamma,\gamma')$ and $P_\urcorner(\gamma,\gamma')$. The dashed lines indicate along which directions the symbols are identified.}
    \label{figure.p_corner}
\end{figure}

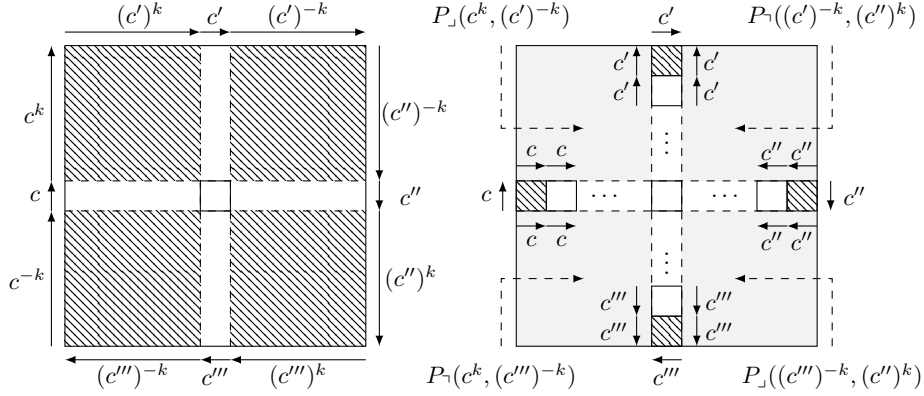
\begin{figure}[!ht]
    \centering
    \begin{tikzpicture}[scale=0.4]
\draw (0,0) rectangle (10,10);
\draw (4.5,4.5) rectangle (5.5,5.5);
            \draw[dashed] (0,4.5) -- (10,4.5);
            \draw[dashed] (4.5,0) -- (4.5,10);
            \draw[dashed] (0,5.5) -- (10,5.5);
            \draw[dashed] (5.5,0) -- (5.5,10);
            \fill[pattern=north west lines] (0,0) rectangle (4.5,4.5);
            \fill[pattern=north west lines] (5.5,5.5) rectangle (10,10);
            \fill[pattern=north west lines] (0,10) rectangle (4.5,5.5);
            \fill[pattern=north west lines] (10,0) rectangle (5.5,4.5);
            \node[scale=0.9] at (2.25,-1) {$(c''')^{-k}$};
            \node[scale=0.9] at (5,-1) {$c'''$};
            \node[scale=0.9] at (7.75,-1) {$(c''')^k$};
            \draw[-latex] (4.5,-0.45) -- (0,-0.45);
            \draw[-latex] (5.5,-0.45) -- (4.5,-0.45);
            \draw[-latex] (10,-0.45) -- (5.5,-0.45);
            \draw[-latex] (10.45,4.5) -- (10.45,0);
            \node[scale=0.9] at (11.5,2.25) {$(c'')^k$};
            \draw[-latex] (10.45,5.5) -- (10.45,4.5);
            \draw[-latex] (10.45,10) -- (10.45,5.5);
            \node[scale=0.9] at (11.5,5) {$c''$};
            \node[scale=0.9] at (11.625,7.75) {$(c'')^{-k}$};
            \draw[latex-] (10,10.45) -- (5.5,10.45);
            \draw[latex-] (5.5,10.45) -- (4.5,10.45);
            \draw[latex-] (4.5,10.45) -- (0,10.45);
            \draw[-latex] (-0.45,0) -- (-0.45,4.5);
            \draw[-latex] (-0.45,4.5) -- (-0.45,5.5);
            \draw[-latex] (-0.45,5.5) -- (-0.45,10);
            \node[scale=0.9] at (7.75,11) {$(c')^{-k}$};
            \node[scale=0.9] at (5,11) {$c'$};
            \node[scale=0.9] at (2.5,11) {$(c')^k$};
            \node[scale=0.9] at (-1,7.75) {$c^k$};
            \node[scale=0.9] at (-1,5) {$c$};
            \node[scale=0.9] at (-1.25,2.25) {$c^{-k}$};
        \begin{scope}[xshift=15cm]
        \fill[gray!10] (0,0) rectangle (4.5,4.5);
        \fill[gray!10] (10,10) rectangle (5.5,5.5);
        \fill[gray!10] (10,0) rectangle (5.5,4.5);
        \fill[gray!10] (0,10) rectangle (4.5,5.5);
        \draw (0,0) rectangle (10,10);
            \draw (4.5,4.5) rectangle (5.5,5.5);
            \draw[dashed] (0,4.5) -- (10,4.5);
            \draw[dashed] (4.5,0) -- (4.5,10);
            \draw[dashed] (0,5.5) -- (10,5.5);
            \fill[pattern = north west lines] (0,4.5) rectangle (1,5.5);
            \fill[pattern = north west lines] (4.5,0) rectangle (5.5,1);
            \fill[pattern = north west lines] (4.5,9) rectangle (5.5,10);
            \fill[pattern = north west lines] (9,4.5) rectangle (10,5.5);
            \draw (0,4.5) rectangle (1,5.5);
            \draw (1,4.5) rectangle (2,5.5);
            \node[scale=0.9] at (3,5) {$\dots$};
            \node[scale=0.9] at (7,5) {$\dots$};
            \node[scale=0.9] at (5,3) {$\vdots$};
            \node[scale=0.9] at (5,7) {$\vdots$};
            \draw[dashed] (5.5,0) -- (5.5,10);
            \draw (4.5,9) rectangle (5.5,10);
            \draw (4.5,8) rectangle (5.5,9);
            \draw (9,4.5) rectangle (10,5.5);
            \draw (8,4.5) rectangle (9,5.5);
            \draw (4.5,0) rectangle (5.5,1);
            \draw (4.5,1) rectangle (5.5,2);
            \draw[latex-] (1,4) -- (0,4);
            \draw[latex-] (2,4) -- (1,4);
            \draw[-latex] (-0.45,4.5) -- (-0.45,5.5);
            \draw[-latex] (0,6) -- (1,6);
            \draw[-latex] (1,6) -- (2,6);
            \draw[-latex] (4,9) -- (4,10);
            \draw[-latex] (4,8) -- (4,9);
            \draw[-latex] (4.5,10.45) -- (5.5,10.45);
            \draw[latex-] (6,10) -- (6,9);
            \draw[latex-] (6,9) -- (6,8);
            \draw[latex-] (8,6) -- (9,6);
            \draw[latex-] (9,6) -- (10,6);
            \draw[-latex] (10.45,5.5) -- (10.45,4.5);
            \draw[-latex] (10,4) -- (9,4);
            \draw[-latex] (9,4) -- (8,4);
            \draw[-latex] (6,2) -- (6,1);
            \draw[latex-] (6,0) -- (6,1);
            \draw[-latex] (5.5,-0.45) -- (4.5,-0.45);
            \draw[latex-] (4,0) -- (4,1);
            \draw[latex-] (4,1) -- (4,2);
            \node[scale=0.9] at (-0.5,-1) {$P_\urcorner(c^k,(c''')^{-k})$};
            \draw[dashed,-latex] (-0.5,0) -- (-0.5,2.25) -- (2.25,2.25);
            \node[scale=0.9] at (-0.5,11) {$P_\lrcorner(c^k, (c')^{-k})$};
            \draw[dashed,-latex] (-0.5,10) -- (-0.5,7.25) -- (2.25,7.25);
            \node[scale=0.9] at (10.5,11) {$P_\urcorner((c')^{-k},(c'')^{k})$};
            \draw[dashed,-latex] (10.5,10) -- (10.5,7.25) -- (7.25,7.25);
            \node[scale=0.9] at (10.5,-1) {$P_\lrcorner((c''')^{-k},(c'')^{k})$};
            \draw[dashed,-latex] (10.5,0) -- (10.5,2.25) -- (7.25,2.25);
            \node[scale=0.9] at (-1,5) {$c$};
            \node[scale=0.9] at (0.5,6.5) {$c$};
            \node[scale=0.9] at (1.5,6.5) {$c$};
            \node[scale=0.9] at (0.5,3.5) {$c$};
            \node[scale=0.9] at (1.5,3.5) {$c$};
            \node[scale=0.9] at (3.5,8.5) {$c'$};
            \node[scale=0.9] at (3.5,9.5) {$c'$};
            \node[scale=0.9] at (5,11) {$c'$};
            \node[scale=0.9] at (6.5,8.5) {$c'$};
            \node[scale=0.9] at (6.5,9.5) {$c'$};
            \node[scale=0.9] at (8.5,6.5) {$c''$};
            \node[scale=0.9] at (9.5,6.5) {$c''$};
            \node[scale=0.9] at (8.5,3.5) {$c''$};
            \node[scale=0.9] at (9.5,3.5) {$c''$};
            \node[scale=0.9] at (11.25,5) {$c''$};

            \node[scale=0.9] at (5,-1) {$c'''$};
            \node[scale=0.9] at (6.75,0.5) {$c'''$};
            \node[scale=0.9] at (6.75,1.5) {$c'''$};
            \node[scale=0.9] at (3.25,1.5) {$c'''$};
            \node[scale=0.9] at (3.25,0.5) {$c'''$};
       \end{scope}
\end{tikzpicture}
    \caption{\label{figure.completion.border}Illustration of the proof of Proposition~\ref{proposition.completion.border}.
        }
\end{figure}
We define a pattern on $\textbf{B}^2(n(2k+1)) \backslash \textbf{B}^2(n-1)$ as illustrated on Figure~\ref{figure.completion.border}. Formally, its restriction on $\llbracket - n(2k+1) , -n\rrbracket \times \llbracket - n(2k+1) , -n\rrbracket$ (resp. $\llbracket - n(2k+1) , -n\rrbracket \times \llbracket n , n(2k+1)\rrbracket$, $\llbracket n , n(2k+1) \rrbracket \times \llbracket n , n(2k+1)\rrbracket$, $\llbracket n , n(2k+1) \rrbracket \times \llbracket -n(2k+1) , -n\rrbracket$ ) to be $P_\urcorner(c^k,(c''')^{-k})$ (resp. $P_\lrcorner(c^{k},(c')^{-k})$, $P_\urcorner((c')^{-k},(c'')^{k})$, $P_\lrcorner((c''')^{-k},(c'')^k)$). The remainder is filled as follows: on each of the squares $\boldsymbol{B}^2(n) + 2nk' \textbf{e}$, where $\textbf{e} = (-1,0)$ (resp. $(0,1), (1,0), (0,-1)$), the restriction is $P_\lrcorner(c,c^{-1})$ (resp. $P_\lrcorner(c',(c')^{-1})$, $P_\lrcorner((c'')^{-1},c'')$, $P_\lrcorner((c''')^{-1},c''')$) for all $k' \in \llbracket 1, k \rrbracket$. This pattern satisfies the conditions of the statement.
\end{proof}

\begin{theorem}\label{theorem.boxext}
For every graph $G$, if the two-dimensional homshift $X_G^2$ has the box-extension property, then the even square group of $G$ is trivial.
\end{theorem}

\begin{proof}
    Let us assume that $X^2_G$ has the box-extension property with parameter $r \ge 1$. Let us prove that every even cycle $c$ of $G$ is square-decomposable. Pick $t$ an arbitrary cycle of length two beginning at $c_0$, and some integer $l$ such that $l \cdot l(c) \ge r$. Apply Proposition~\ref{proposition.completion.border} on the cycles $c$ and $c' = c'' = c'''= t^{l(c)/2}$ to get a annular pattern on $\textbf{B}^2(n(2k+1)) \backslash \textbf{B}^2(n-1)$ whose restriction to the inner boundary $\partial\textbf{B}^2(n)$ is $c\odot t^{3l(c)/2}$. By the box-extension property, this pattern can be extended to $\textbf{B}^2(n(2k+1))$. By Lemma~\ref{lemma.border.sqdec}, the cycle $c\odot t^{3l(c)/2}$, and therefore $c$, is square-decomposable.
\end{proof}

\subsection{\label{section.nonboxext}A counterexample for the converse implication}

It is immediate that the graph $K$ shown on Figure~\ref{figure.kenkatabami} is square-decomposable. In particular, its even square group is trivial. However, it does not have the box-extension property.

\begin{proposition}
\label{prop.no.boxext}
    Consider a graph $G$ which has a {non-backtracking} cycle $c$ of even length on $G$ such that for all $i$, $c_{i}$ is the only common neighbor of {$c_{(i-1)\bmod l(c)}$ and  $c_{(i+1)\bmod l(c)}$} in $G$. The homshift $X^2_G$ does not have the box-extension property.
\end{proposition}

\begin{proof}
    The stated property of $c$ implies that for all $l \ge 0$, there is a unique pattern on  $\textbf{B}^2(n(2l+1)) \backslash \textbf{B}^2(n-1)$
    whose restriction $p$ to $\partial \textbf{B}^2(n(2l+1))$ is {defined by $p_l = c^{-l} \odot c^{l+1}$, $p_u = {c}^{l+1} \odot {c}^{-l}$, $p_r = {c}^{-l} \odot {c}^{l+1}$ and $p_d = {c}^{l+1} \odot {c}^{-l}$}, and this is the pattern provided by Proposition~\ref{proposition.completion.border}. Assume for a contradiction that $X^2_G$ has the box-extension property with parameter $r \ge 1$. Take $l$ such that $l \cdot l(c) \ge r$. {By application of the box-extension property on the pattern on $\textbf{B}^2(n(2l+1)) \backslash \textbf{B}^2(n(2l+1)-1)$, we obtain an extension of it on $\textbf{B}^2(n(2l+1))$. By uniqueness of the extension on $\textbf{B}^2(n(2l+1)) \backslash \textbf{B}^2(n-1)$, this extension must coincide on this set with the pattern provided by Proposition~\ref{proposition.completion.border}. This implies that the restriction of this extension to 
    $\partial \textbf{B}^2(n)$ must be $c^4$.}
However, one can check, using the hypothesis on the cycle $c$, that there is no locally admissible extension of this pattern to $\textbf{B}^2(n(2l+1)) \backslash \textbf{B}^2(n-2)$.
\end{proof}

\begin{corollary}
    The subshift $X_K$ does not have the box-extension property.
\end{corollary}

\begin{proof}
    The exterior cycle of $K$ \[c\coloneqq \epsilon_1 \gamma_1 \epsilon_2 \gamma_2 \epsilon_3 \gamma_3 \epsilon_1\] satisfies the condition of Proposition~\ref{prop.no.boxext}.
\end{proof}

{Let us end this section with the following question that we leave open:}

\begin{question}
    How to characterize graphs $G$ for which $X^2_G$ has the box-extension property? 
\end{question}

\begin{question}
    Does strong irreducibility imply cohomological triviality for homshifts?
\end{question}

\section{\label{section.non.mixing}Cohomology of non-mixing homshifts}

Throughout this article we have worked exclusively on mixing subshifts, for the simple reason that the techniques that we use do not generalize to non-mixing subshifts. As an example, we exhibit here a nontrivial cocycle on the homshift corresponding to the graph on two vertices connected by a single edge, which has trivial square group (and thus even square group).  This homshift is non-mixing, since this graph is bipartite. This is formulated as Proposition \ref{lemma.non-mixing}.

Recall that a continuous cocycle $c$ on a subshift $X$ with values in $\mathbb G$ is said to be trivial when there exists a continuous function $b : X \rightarrow \mathbb G$ and a group homomorphism $\phi : \mathbb Z^d \rightarrow \mathbb G$ such that for all $\boldsymbol n$ and $x \in X$, $c(\boldsymbol n , x) = {b(\sigma^{\boldsymbol{n}}(x))}^{-1} \phi(\boldsymbol{n}) b(x).$

There is a possibility that this condition is not satisfied for any maps $b$ and $\phi$ with values in $\mathbb G$, but is satisfied for some maps with values in a larger group $\mathbb G'$.
This situation does arise for the example that we present. We denote by $\mathbb{F}_2$ the free group with two generators, denoted by $\alpha$ and $\beta$.

\begin{proposition}\label{lemma.non-mixing}
    Let $H$ be the graph with vertex set $\{0,1\}$ and one edge between $0$ and $1$. The homshift $X^2_H$ has an $\mathbb F_2$-valued continuous cocycle which is not continuously cohomologous to a group homomorphism. However it is cohomologous to a homomorphism for a transfer function valued in a larger group $\mathbb G \supset \mathbb F_2$.
\end{proposition}

\begin{proof}
   
    The subshift $X_H$ contains exactly two elements $x$ and $y$, such that $x_{\boldsymbol{0}}=0$ and $y_{\boldsymbol{0}}=1$. 
    Let us consider the continuous cocycle $c : \mathbb{Z}^2 \times X_H \rightarrow \mathbb{F}_2$ which is defined by $c(\boldsymbol{e}^1,x) = c(\boldsymbol{e}^2,x) = \alpha$ and $c(\boldsymbol{e}^1,y) = c(\boldsymbol{e}^2,y) = \beta$. 
    
    The cocycle $c$ cannot be cohomologous to a group homomorphism. Indeed, let us assume for a contradiction that there exists a group homomorphism $c'$ and $b : X_H \rightarrow \mathbb{F}_2$ such that for all $\boldsymbol{n}$ and for all $z$,
    \[{c(\boldsymbol{n},z) ={b(\sigma^{\boldsymbol{n}}(z))}^{-1} c'(\boldsymbol{n}) b(z).}\]
    In particular,
    \[
    c(\boldsymbol{e}^1,x) ={b(\sigma^{\boldsymbol{e}^1}(x))}^{-1} c'(\boldsymbol{e}^1) b(x),
    \]
    \[
    c(\boldsymbol{e}^1,y) ={b(\sigma^{\boldsymbol{e}^1}(y))}^{-1} c'(\boldsymbol{e}^1) b(y).
    \]
    Since 
    $\sigma^{\boldsymbol{e}^1}(x) = y$ and $\sigma^{\boldsymbol{e}^1}(y) = x$, these two equations yield
    \[
    b(x)^{-1}b(y) c(\boldsymbol{e}^1,x)b(x)^{-1}b(y)  = c(\boldsymbol{e}^1,y) 
    \]

Consider the morphism $\eta:\mathbb F_2\to \Z$ defined by $\alpha \mapsto 0$, $\beta \mapsto 1$. Applying $\eta$ to the previous equation yields $2\eta(b(x)^{-1} b(y)) = 1$, which is clearly impossible.
    
    On the other hand, consider the group $\mathbb G= \langle \alpha,\beta,w : w\alpha w=\beta \rangle$. Notice that $\mathbb G$ is a free group modulo a single relation. One can easily deduce that the map $\iota : \mathbb F_2 \rightarrow \mathbb G$ such that $\iota(\alpha) = \alpha$ and $\iota(\beta) = \beta$ is an embedding. See \cite[Chapter II, Proposition 5.1]{Lyndon}.

    Define a function $b$ by setting $b(x)=w^{-1}$ and $b(y)=1_{\mathbb G}$ and a group homomorphism $h : \mathbb Z^2\to \mathbb G$ by
    \[h(\boldsymbol{e}^1)=h(\boldsymbol{e}^2)=\alpha w\] 
    and the cocycle $c'$ given by
    $c'(\boldsymbol n, \cdot)= h(\boldsymbol{n})$. A simple computation provides that $c$ is cohomologous to $c'$ for the transfer function $b$.
\end{proof}

\begin{remark}
    When the even square group of $G$ is not trivial, the square group cocycles of $X^d_G$ are not trivial, even when allowing a transfer function valued in a larger group. The proof is essentially the same as the one of Proposition \ref{proposition.non.trivial.sqgrp.cocycle}.
\end{remark}

\printindex

\bibliographystyle{alphaurl}
\bibliography{biblio.bib}

\end{document}